\documentclass[12pt,a4paper,reqno]{amsart}
\title[Unbounded twisted complexes]{Unbounded twisted complexes}
\author{Rina Anno}
\email{ranno@math.ksu.edu}
\address{Department of Mathematics \\
Kansas State University \\
138 Cardwell Hall \\
Manhattan, KS 66506\\
USA}
\author{Timothy Logvinenko} 
\email{LogvinenkoT@cardiff.ac.uk} 
\address{School of Mathematics\\ 
Cardiff University\\
Senghennydd Road,\\
Cardiff, CF24 4AG\\
UK}
\usepackage{amsmath,amsfonts,amssymb,amsthm,epsfig,amscd,latexsym,comment}
\usepackage{caption}
\usepackage{MnSymbol}
\usepackage{tikz}
\usetikzlibrary{cd}
\usepackage{graphicx}
\usepackage{array}
\usepackage{subfigure}
\usepackage{leftidx}
\usepackage{xparse}% http://ctan.org/pkg/xparse

\usepackage[colorlinks=true, pdfpagemode=none, pdfmenubar=false, linkcolor=blue, citecolor=blue, urlcolor=blue]{hyperref}

\let\amsamp=&

\begingroup
\catcode`\&=13
\gdef\smallampmatrix{%
  \begingroup
  \let&=\amsamp
  \begin{smallmatrix}%
}
\gdef\endsmallampmatrix{\end{smallmatrix}\endgroup}
\endgroup

\DeclareMathOperator{\obj}{Ob}

\DeclareMathOperator{\homm}{Hom}

\DeclareMathOperator{\eend}{End}

\DeclareMathOperator{\picr}{Pic}

\DeclareMathOperator{\cl}{Cl}

\DeclareMathOperator{\action}{act}

\DeclareMathOperator{\modd}{\bf Mod}

\DeclareMathOperator{\id}{Id}

\DeclareMathOperator{\vectspaces}{\bf Vect}

\DeclareMathOperator{\opp}{{opp}}
\DeclareMathOperator{\fg}{{\it fg}}
\DeclareMathOperator{\qrep}{\it \mathcal{Q}r}
\DeclareMathOperator{\hproj}{\mathcal{P}}

\DeclareMathOperator{\semifree}{\mathcal{S}\mathcal{F}}
\DeclareMathOperator{\sffg}{\mathcal{S}\mathcal{F}_{\fg}}
\DeclareMathOperator{\perf}{{\it \mathcal{P}erf}}
\DeclareMathOperator{\hperf}{{\it h\mathcal{P}erf}}
\DeclareMathOperator{\hmtpy}{{Ho}}

\DeclareMathOperator{\tria}{{Tria}}
\DeclareMathOperator{\twcx}{{Tw}}
\DeclareMathOperator{\twbicx}{{Twbi}}
\DeclareMathOperator{\pretriag}{{Pre\text{-}Tr}}
\DeclareMathOperator{\DGFun}{{DGFun}}

\DeclareMathOperator{\TPair}{{\bf TPair}}
\DeclareMathOperator{\alg}{{\bf Alg}}

\DeclareMathOperator{\conv}{{Conv}}

\DeclareMathOperator{\free}{{\bf Free}}

\DeclareMathOperator{\eilmoor}{{\mathcal{E}}}
\DeclareMathOperator{\coeilmoor}{{co\mathcal{E}}}
\DeclareMathOperator{\kleisli}{{\mathcal{K}\mathcal{S}}}
\DeclareMathOperator{\cokleisli}{{co\mathcal{K}\mathcal{S}}}
\DeclareMathOperator{\cxrow}{{\mathcal{C}xrow}}
\DeclareMathOperator{\cxcol}{{\mathcal{C}xcol}}

\DeclareMathOperator{\Hom}{Hom}

\begin{document}

\def\bv{\mathbf{v}}
\def\kgc_{K^*_G(\mathbb{C}^n)}
\def\kgchi_{K^*_\chi(\mathbb{C}^n)}
\def\kgcf_{K_G(\mathbb{C}^n)}
\def\kgchif_{K_\chi(\mathbb{C}^n)}
\def\gpic_{G\text{-}\picr}
\def\gcl_{G\text{-}\cl}
\def\trch_{{\chi_{0}}}
\def\regring{{R}}
\def\regrep{{V_{\text{reg}}}}
\def\givrep{{V_{\text{giv}}}}
\def\lbar{{(\mathbb{Z}^n)^\vee}}
\def\genpx_{{p_X}}
\def\genpy_{{p_Y}}
\def\genpcn_{p_{\mathbb{C}^n}}
\def\gnat{gnat}
\def\twalg{{\regring \rtimes G}}
\def\L{{\mathcal{L}}}
\def\gcd{\mbox{gcd}}
\def\lcm{\mbox{lcm}}
\def\tf{{\tilde{f}}}
\def\tD{{\tilde{D}}}
\def\A{{\mathcal{A}}}
\def\B{{\mathcal{B}}}
\def\C{{\mathcal{C}}}
\def\D{{\mathcal{D}}}
\def\E{{\mathcal{E}}}
\def\F{{\mathcal{F}}}
\def\H{{\mathcal{H}}}
\def\L{{\mathcal{L}}}
\def\M{{\mathcal{M}}}
\def\N{{\mathcal{N}}}
\def\R{{\mathcal{R}}}
\def\T{{\mathcal{T}}}
\def\RF{{\mathcal{R}\mathcal{F}}}
\def\barA{{\bar{\mathcal{A}}}}
\def\barAi{{\bar{\mathcal{A}}_1}}
\def\barAj{{\bar{\mathcal{A}}_2}}
\def\barB{{\bar{\mathcal{B}}}}
\def\barC{{\bar{\mathcal{C}}}}
\def\barD{{\bar{\mathcal{D}}}}
\def\barT{{\bar{\mathcal{T}}}}
\def\barM{{\bar{\mathcal{M}}}}
\def\Aopp{{\A^{\opp}}}
\def\Bopp{{\B^{\opp}}}
\def\Copp{{\C^{\opp}}}
\def\aA{\leftidx{_{a}}{\A}}
\def\bA{\leftidx{_{b}}{\A}}
\def\Aa{{\A_a}}
\def\Ea{E_a}
\def\aE{\leftidx{_{a}}{E}{}}
\def\Eb{E_b}
\def\bE{\leftidx{_{b}}{E}{}}
\def\Fa{F_a}
\def\aF{\leftidx{_{a}}{F}{}}
\def\Fb{F_b}
\def\bF{\leftidx{_{b}}{F}{}}
\def\aM{\leftidx{_{a}}{M}{}}
\def\aN{\leftidx{_{a}}{N}{}}
\def\aMb{\leftidx{_{a}}{M}{_{b}}}
\def\aNb{\leftidx{_{a}}{N}{_{b}}}
\def\rfRFa{\leftidx{_{\RF}}{\RF}{_{\A}}}
\def\aRFrf{\leftidx{_{\A}}{\RF}{_{\RF}}}
\def\biAMA{\leftidx{_{\A}}{M}{_{\A}}}
\def\biAMC{\leftidx{_{\A}}{M}{_{\C}}}
\def\biCMA{\leftidx{_{\C}}{M}{_{\A}}}
\def\biCMC{\leftidx{_{\C}}{M}{_{\C}}}
\def\biALA{\leftidx{_{\A}}{L}{_{\A}}}
\def\biALC{\leftidx{_{\A}}{L}{_{\C}}}
\def\biCLA{\leftidx{_{\C}}{L}{_{\A}}}
\def\biCLC{\leftidx{_{\C}}{L}{_{\C}}}
\def\Na{{N_a}}
\def\vectk{{\vectspaces\text{-}k}}
\def\vectkfg{{\vectspaces_{\text{fg}}\text{-}k}}
\def\modk{{\modd\text{-}k}}
\def\Amod{{\A\text{-}\modd}}
\def\modA{{\modd\text{-}\A}}
\def\modbar{{\overline{\modd}}}
\def\modbarA{{\overline{\modd}\text{-}\A}}
\def\modbarAopp{{\overline{\modd}\text{-}\Aopp}}
\def\modB{{\modd\text{-}\B}}
\def\modC{{\modd\text{-}\C}}
\def\modD{{\modd\text{-}\D}}
\def\modbarB{{\overline{\modd}\text{-}\B}}
\def\modbarC{{\overline{\modd}\text{-}\C}}
\def\modbarD{{\overline{\modd}\text{-}\D}}
\def\modbarBopp{{\overline{\modd}\text{-}\Bopp}}
\def\kmodk{{k\text-\modd\text{-}k}}
\def\bikmodk{{k_\bullet\text-\modd\text{-}k_\bullet}}
\def\AmodA{{\A\text{-}\modd\text{-}\A}}
\def\AmodM{{\A\text{-}\modd\text{-}\M}}
\def\AmodB{{\A\text{-}\modd\text{-}\B}}
\def\AmodT{{\A\text{-}\modd\text{-}\T}}
\def\BmodB{{\B\text{-}\modd\text{-}\B}}
\def\BmodA{{\B\text{-}\modd\text{-}\A}}
\def\DmodD{{\D\text{-}\modd\text{-}\D}}
\def\MmodA{{\M\text{-}\modd\text{-}\A}}
\def\MmodM{{\M\text{-}\modd\text{-}\M}}
\def\TmodA{{\T\text{-}\modd\text{-}\A}}
\def\TmodT{{\T\text{-}\modd\text{-}\T}}
\def\AmodbarA{\A\text{-}{\overline{\modd}\text{-}\A}}
\def\AmodbarB{\A\text{-}{\overline{\modd}\text{-}\B}}
\def\AmodbarC{\A\text{-}{\overline{\modd}\text{-}\C}}
\def\AmodbarD{\A\text{-}{\overline{\modd}\text{-}\D}}
\def\AmodbarM{\A\text{-}{\overline{\modd}\text{-}\M}}
\def\AmodbarT{\A\text{-}{\overline{\modd}\text{-}\T}}
\def\BmodbarA{\B\text{-}{\overline{\modd}\text{-}\A}}
\def\BmodbarB{\B\text{-}{\overline{\modd}\text{-}\B}}
\def\BmodbarC{\B\text{-}{\overline{\modd}\text{-}\C}}
\def\BmodbarD{\B\text{-}{\overline{\modd}\text{-}\D}}
\def\CmodbarA{\C\text{-}{\overline{\modd}\text{-}\A}}
\def\CmodbarB{\C\text{-}{\overline{\modd}\text{-}\B}}
\def\CmodbarC{\C\text{-}{\overline{\modd}\text{-}\C}}
\def\CmodbarD{\C\text{-}{\overline{\modd}\text{-}\D}}
\def\DmodbarA{\D\text{-}{\overline{\modd}\text{-}\A}}
\def\DmodbarB{\D\text{-}{\overline{\modd}\text{-}\B}}
\def\DmodbarC{\D\text{-}{\overline{\modd}\text{-}\C}}
\def\DmodbarD{\D\text{-}{\overline{\modd}\text{-}\D}}
\def\TmodbarA{\T\text{-}{\overline{\modd}\text{-}\A}}
\def\MmodbarA{\M\text{-}{\overline{\modd}\text{-}\A}}
\def\MmodbarM{\M\text{-}{\overline{\modd}\text{-}\M}}
\def\modbarT{{\overline{\modd}\text{-}T}}
\def\freeA{{\free\text{-}A}}
\def\freeAS{{\free_S\text{-}A}}
\def\freepfA{{\free_{pf}\text{-}A}}
\def\Afree{{A\text{-}\free}}
\def\AfreeA{{A\text{-}\free\text{-}A}}
\def\AfreeB{{A\text{-}\free\text{-}B}}
\def\sfA{{\semifree(\A)}}
\def\sfB{{\semifree(\B)}}
\def\sffgA{{\sffg(\A)}}
\def\sffgB{{\sffg(\B)}}
\def\hprojA{{\hproj(\A)}}
\def\hprojB{{\hproj(\B)}}
\def\qrepA{{\qrep(\A)}}
\def\qrepB{{\qrep(\B)}}
\def\opp{{\text{opp}}}
\def\Aperf{{\A^{\text{pf}}}}
\def\hperfA{{\hperf(\A)}}
\def\hperfB{{\hperf(\B)}}
\def\barperf{{\it\mathcal{P}\overline{er}f}}
\def\barperfA{{\barperf\text{-}\A}}
\def\barperfB{{\barperf\text{-}\B}}
\def\barperfC{{\barperf\text{-}\C}}
\def\qrhpr{{\hproj^{qr}}}
\def\qrhprA{{\qrhpr(\A)}}
\def\qrhprB{{\qrhpr(\B)}}
\def\qrsf{{\semifree^{qr}}}
\def\qrsf{{\semifree^{qr}}}
\def\qrsfA{{\qrsf(\A)}}
\def\qrsfB{{\qrsf(\B)}}
\def\Aperfsf{{\semifree^{\A\text{-}\perf}(\AbimB)}}
\def\Bperfsf{{\semifree^{\B\text{-}\perf}(\AbimB)}}
\def\Aprfhpr{{\hproj^{\A\text{-}\perf}(\AbimB)}}
\def\Bprfhpr{{\hproj^{\B\text{-}\perf}(\AbimB)}}
\def\Aqrhpr{{\hproj^{\A\text{-}qr}(\AbimB)}}
\def\Bqrhpr{{\hproj^{\B\text{-}qr}(\AbimB)}}
\def\Aqrsf{{\semifree^{\A\text{-}qr}(\AbimB)}}
\def\Bqrsf{{\semifree^{\B\text{-}qr}(\AbimB)}}
\def\modAopp{{\modd\text{-}\Aopp}}
\def\modBopp{{\modd\text{-}\Bopp}}
\def\AmodA{{\A\text{-}\modd\text{-}\A}}
\def\AmodB{{\A\text{-}\modd\text{-}\B}}
\def\AmodC{{\A\text{-}\modd\text{-}\C}}
\def\BmodA{{\B\text{-}\modd\text{-}\A}}
\def\BmodB{{\B\text{-}\modd\text{-}\B}}
\def\BmodC{{\B\text{-}\modd\text{-}\C}}
\def\CmodA{{\C\text{-}\modd\text{-}\A}}
\def\CmodB{{\C\text{-}\modd\text{-}\B}}
\def\CmodC{{\C\text{-}\modd\text{-}\C}}
\def\CmodD{{\C\text{-}\modd\text{-}\D}}
\def\AbimA{{\A\text{-}\A}}
\def\AbimC{{\A\text{-}\C}}
\def\AbimM{{\A\text{-}\M}}
\def\BbimA{{\B\text{-}\A}}
\def\BbimB{{\B\text{-}\B}}
\def\BbimC{{\B\text{-}\C}}
\def\BbimD{{\B\text{-}\D}}
\def\CbimA{{\C\text{-}\A}}
\def\CbimB{{\C\text{-}\B}}
\def\CbimC{{\C\text{-}\C}}
\def\DbimA{{\D\text{-}\A}}
\def\DbimB{{\D\text{-}\B}}
\def\DbimC{{\D\text{-}\C}}
\def\DbimD{{\D\text{-}\D}}
\def\MbimA{{\M\text{-}\A}}
\def\AhprA{{\hproj\left(\AbimA\right)}}
\def\BhprB{{\hproj\left(\BbimB\right)}}
\def\AhprB{{\hproj\left(\AbimB\right)}}
\def\BhprA{{\hproj\left(\BbimA\right)}}
\def\AbarA{{\overline{\A\text{-}\A}}}
\def\AbarB{{\overline{\A\text{-}\B}}}
\def\BbarA{{\overline{\B\text{-}\A}}}
\def\BbarB{{\overline{\B\text{-}\B}}}
\def\QAbimB{{Q\A\text{-}\B}}
\def\AbimB{{\A\text{-}\B}}
\def\AbimC{{\A\text{-}\C}}
\def\AonebimB{{\A_1\text{-}\B}}
\def\AtwobimB{{\A_2\text{-}\B}}
\def\BbimA{{\B\text{-}\A}}
\def\MddA{{M^{\tilde{\A}}}}
\def\MddB{{M^{\tilde{\B}}}}
\def\MhdA{{M^{h\A}}}
\def\MhdB{{M^{h\B}}}
\def\NhdB{{N^{h\B}}}
\def\Cat{{\it \mathcal{C}at}}
\def\twoCat{{\it 2\text{-}\;\mathcal{C}at}}
\def\DGCat{{DG\text{-}Cat}}
\def\HoDGCat{{\hmtpy(\DGCat)}}
\def\HoDGCatV{{\hmtpy(\DGCat_\mathbb{V})}}
\def\MoDGCat{{Mo(\DGCat)}}
\def\tr{{tr}}
\def\pretr{{pretr}}
\def\kctr{{kctr}}
\def\PreTrCat{{\DGCat^\pretr}}
\def\KcTrCat{{\DGCat^\kctr}}
\def\HoPretrCat{{\hmtpy(\PreTrCat)}}
\def\HoKcTrCat{{\hmtpy(\KcTrCat)}}
\def\Aquasirep{{\A\text{-}qr}}
\def\QAquasirep{{Q\A\text{-}qr}}
\def\Bquasirep{{\B\text{-}qr}} 
\def\lderA{{\tilde{\A}}} 
\def\lderB{{\tilde{\B}}} 
\def\adjunit{{\text{adj.unit}}}
\def\adjcounit{{\text{adj.counit}}}
\def\degzero{{\text{deg.0}}}
\def\degone{{\text{deg.1}}}
\def\degminusone{{\text{deg.-$1$}}}
\def\barzeta{{\overline{\zeta}}}
\def\Ract{{R {\action}}}
\def\barRact{{\overline{\Ract}}}
\def\actL{{{\action} L}}
\def\baractL{{\overline{\actL}}}
\def\Ainfty{{A_{\infty}}}
\def\moddinf{{{\bf Mod}_{\infty}}}
\def\nodd{{{\bf Nod}}}
\def\noddinf{{{\bf Nod}_{\infty}}}
\def\perfinf{{{\bf Perf}_{\infty}}}
\def\conodd{{{\bf coNod}}}
\def\comodd{{{\bf coMod}}}
\def\conoddinf{{{\bf coNod}_{\infty}}}
\def\conodddgstrict{{{\bf coNod}_{dg}^{strict}}}
\def\conodddg{{{\bf coNod}_{dg}}}
\def\conodddghu{{{\bf coNod}_{dg}^{hu}}}
\def\conoddinfshu{{{\bf coNod}_{\infty}^{shu}}}
\def\noddinfstr{{{\bf Nod}^{\text{strict}}_{\infty}}}
\def\noddinfA{{\noddinf\A}}
\def\noddinfB{{\noddinf\B}}
\def\perfinfA{{\perfinf\A}}
\def\perfinfB{{\perfinf\B}}
\def\nodA{{\noddinf\text{-}A}}
\def\nodstrA{{\nodd\text{-}A}}
\def\nodstrhuA{{\nodd^{hu}\text{-}A}}
\def\nodhuA{{\noddinfhu\text{-}A}}
\def\nodhupfA{{\noddinfhupf\text{-}A}}
\def\nodB{{\noddinf\text{-}B}}
\def\nodAbic{{\left(\nodA\right)^{\text{bicat}}}}
\def\Anodbic{{\left(\Anod\right)^{\text{bicat}}}}
\def\nodBbic{{\left(\nodB\right)^{\text{bicat}}}}
\def\Bnodbic{{\left(\Bnod\right)^{\text{bicat}}}}
\def\repA{{A^{\B}}}
\def\nodrepA{{\noddinf\text{-}\repA}}
\def\pretrA{{A^{\pretriag\A}}}
\def\nodpretrA{{\noddinf\text{-}\pretrA}}
\def\twcxA{{A^{\twcxub\A}}}
\def\nodtwcxA{{\noddinf\text{-}\twcxA}}
\def\twcxrepA{{A^{\twcxub\modA}}}
\def\nodtwcxrepA{{\noddinf\text{-}\twcxrepA}}
\def\Anod{{A\text{-}\noddinf}}
\def\Bnod{{B\text{-}\noddinf}}
\def\AnodA{{A\text{-}\noddinf\text{-}A}}
\def\AnodB{{A\text{-}\noddinf\text{-}B}}
\def\noddinfAB{{\noddinf\AbimB}}
\def\noddinfBA{{\noddinf\BbimA}}
\def\noddinfu{{({\bf Nod}_{\infty})_u}}
\def\noddinfuA{{(\noddinfA)_u}}
\def\noddinfhu{{{\bf Nod}_{\infty}^{hu}}}
\def\noddinfhupf{{{\bf Nod}_{\infty}^{hupf}}}
\def\noddinfhuA{{(\noddinfA)_{hu}}}
\def\noddinfhuB{{(\noddinfB)_{hu}}}
\def\noddinfdg{{({\bf Nod}_{\infty})_{dg}}}
\def\noddinfdgA{{(\noddinfA)_{dg}}}
\def\noddinfdgAA{{(\noddinf\AbimA)_{dg}}}
\def\noddinfdgAB{{(\noddinf\AbimB)_{dg}}}
\def\noddinfdgB{{(\noddinfB)_{dg}}}
\def\moddinf{{\modd_{\infty}}}
\def\moddinfA{{\modd_{\infty}\A}}
\def\naug{{\text{na}}}
\def\infbar{B_\infty}
\def\infbarl{{B^{l}_\infty}}
\def\infbarr{{B^{r}_\infty}}
\def\infbarbi{{B^{bi}_\infty}}
\def\infbarnaug{{B^{\naug}_\infty}}
\def\infcobar{{CB_\infty}}
\def\infbarres{{\bar{B}_\infty}}
\def\infcobarres{{C\bar{B}_\infty}}
\def\infbarA{{B^A_\infty}}
\def\infbarB{{B^B_\infty}}
\def\inftimes{{\overset{\infty}{\otimes}}}
\def\infhom{{\overset{\infty}{\homm}}}
\def\barhom{{\rm H\overline{om}}}
\def\barend{{\overline{\eend}}}
\def\bartimes{{\;\overline{\otimes}}}
\def\bartimesA{{\;\overline{\otimes}_\A\;}}
\def\bartimesB{{\;\overline{\otimes}_\B\;}}
\def\bartimesC{{\;\overline{\otimes}_\C\;}}
\def\triaA{{\tria \A}}
\def\TPairdg{{\TPair^{dg}}}
\def\algA{{\alg(\A)}}
\def\Ainfty{{A_{\infty}}}
\def\gpmu{{\boldsymbol{\mu}}}
\def\odd{{\text{odd}}}
\def\even{{\text{even}}}
\def\twcxub{\twcx^{\pm}}
\def\twcxmns{\twcx^{-}}
\def\twcxpls{\twcx^{+}}
\def\twbicxub{\twbicx^{\pm}}
\def\twbicxmns{\twbicx^{-}}
\def\twbicxpls{\twbicx^{+}}
\def\twbios{\twbicx^{os}}
\def\twbiosub{{\twbicx^{\pm,os}}}
\def\twbiosmns{{\twbicx^{-,os}}}
\def\twbiospls{{\twbicx^{+,os}}}
\def\pretriagub{{\pretriag^{\pm}}}
\def\pretriagmns{{\pretriag^{-}}}
\def\pretriagpls{{\pretriag^{+}}}
\def\eilmoordg{\eilmoor^{\text{dg}}}
\def\hprojemdg{\hproj^{\text{dg}}}
\def\hperfemdg{\hperf^{\text{dg}}}
\def\coeilmoordg{\coeilmoor^{\text{dg}}}
\def\kleisliA{{\kleisli(A)}}
\def\kleisliAS{{\kleisli_S(A)}}
\def\kleisliwk{\kleisli^{\text{wk}}}
\def\eilmoorwk{\eilmoor^{\text{wk}}}
\def\eilmoorwkpf{\eilmoor^{\text{wk,perf}}}
\def\eilmoorwmor{\eilmoor^{\text{dg,wkmor}}}
\def\coeilmoorwk{\coeilmoor^{\text{wk}}}
\def\coeilmoorwmor{\coeilmoor^{\text{dg,wkmor}}}
\def\dnat{{d_{\text{nat}}}}
\def\bareta{{\tilde{\eta}}}
\def\barepsilon{{\tilde{\epsilon}}}
\def\infalg{{\alg_\infty}}
\def\enhcatkc{{\bf EnhCat_{kc}}}
\def\enhcatkcdg{{\bf EnhCat_{kc}^{dg}}}

\def\conodA{{\conoddinf\text{-}A}}
\def\conodB{{\conoddinf\text{-}B}}
\def\conodC{{\conoddinf\text{-}C}}
\def\conodstrC{{\conodd\text{-}C}}
\def\Aconod{{A\text{-}\conoddinf}}
\def\AconodB{{A\text{-}\conoddinf\text{-}B}}
\def\conodhuA{{{\bf coNod}_{\infty}^{hu}\text{-}A}}
\def\conodshuA{{{\bf coNod}_{\infty}^{shu}\text{-}A}}
\def\conodhuB{{{\bf coNod}_{\infty}^{hu}\text{-}B}}
\def\conodshuB{{{\bf coNod}_{\infty}^{shu}\text{-}B}}
\def\conodhuC{{{\bf coNod}_{\infty}^{hu}\text{-}C}}
\def\conodshuC{{{\bf coNod}_{\infty}^{shu}\text{-}C}}
\def\cokleisliC{{\cokleisli(C)}}
\def\cokleisliCS{{\cokleisli_S(C)}}

\theoremstyle{definition}
\newtheorem{defn}{Definition}[section]
\newtheorem*{defn*}{Definition}
\newtheorem{exmpl}[defn]{Example}
\newtheorem*{exmpl*}{Example}
\newtheorem{exrc}[defn]{Exercise}
\newtheorem*{exrc*}{Exercise}
\newtheorem*{chk*}{Check}
\newtheorem{remark}[defn]{Remark}
\newtheorem*{remark*}{Remark}
\theoremstyle{plain}
\newtheorem{theorem}{Theorem}[section]
\newtheorem*{theorem*}{Theorem}
\newtheorem{conj}[defn]{Conjecture}
\newtheorem*{conj*}{Conjecture}
\newtheorem{question}[defn]{Question}
\newtheorem*{question*}{Question}
\newtheorem{prps}[defn]{Proposition}
\newtheorem*{prps*}{Proposition}
\newtheorem{cor}[defn]{Corollary}
\newtheorem*{cor*}{Corollary}
\newtheorem{lemma}[defn]{Lemma}
\newtheorem*{claim*}{Claim}
\newtheorem{Specialthm}{Theorem}
\renewcommand\theSpecialthm{\Alph{Specialthm}}
\numberwithin{equation}{section}
\renewcommand{\textfraction}{0.001}
\renewcommand{\topfraction}{0.999}
\renewcommand{\bottomfraction}{0.999}
\renewcommand{\floatpagefraction}{0.9}
\setlength{\textfloatsep}{5pt}
\setlength{\floatsep}{0pt}
\setlength{\abovecaptionskip}{2pt}
\setlength{\belowcaptionskip}{2pt}

\begin{abstract}
We define unbounded twisted complexes and bicomplexes
generalising the notion of a (bounded) twisted complex over a DG
category \cite{BondalKapranov-EnhancedTriangulatedCategories}. 
These need to be considered relative to another DG category $\B$ admitting 
countable direct sums and shifts. The resulting DG category of
unbounded twisted complexes has a fully faithful convolution functor into 
$\modB$ which filters through $\B$ if the latter admits change of
differential. As an application, we rewrite definitions of
$\Ainfty$-structures in terms of twisted complexes to make them work
in an arbitrary monoidal DG category or a DG bicategory. 
\end{abstract}

\maketitle

\section{Introduction}
\label{section-introduction}

The notion of a twisted complex of objects in a DG category was
introduced by Bondal and Kapranov 
\cite{BondalKapranov-EnhancedTriangulatedCategories}. It was used
as a tool to study and construct DG enhancements of triangulated categories.  
A one-sided twisted complex over a DG category $\A$ can be thought of 
as a lift to $\A$ of a bounded complex of objects in its homotopy category 
$H^0(\A)$. The lift includes the maps in the complex, the homotopies
up to which the consecutive maps compose to zero, and then the higher
homotopies. In addition to the complex of objects in
$H^0(\A)$, such data specifies a choice of its convolution together with a collection 
of Postnikov systems computing this convolution
\cite[\S2.4]{AnnoLogvinenko-BarCategoryOfModulesAndHomotopyAdjunctionForTensorFunctors}.  
Taking the category of one-sided twisted complexes over $\A$ 
is a DG realisation of taking the triangulated hull of $H^0(\A)$. 
Now twisted complexes are ubiquitous in working with 
DG categories and their modules
\cite{Keller-OnDifferentialGradedCategories}
\cite{Toen-LecturesOnDGCategories}\cite{LuntsOrlov-UniquenessOfEnhancementForTriangulatedCategories}\cite{AnnoLogvinenko-PFunctors}\cite{Barbacovi-OnTheCompositionOfTwoSphericalTwists}.

This paper generalises the notion of a twisted complex  
to include unbounded complexes. The authors came to need it, 
and expected the generalisation to be straightfoward. It 
turned out to involve numerous subtleties. The purpose of this short note 
is to write down these subtleties for the benefit of others. 
We also give the original application we had in mind: rewriting 
the definitions of $\Ainfty$-structures
\cite{Keller-IntroductionToAInfinityAlgebrasAndModules}\cite{Lefevre-SurLesAInftyCategories}\cite{Lyubashenko-CategoryOfAinftyCategories}
in terms of twisted complexes. This decouples them from the differential $m_1$ 
and allows them to work in an arbitrary monoidal DG category. 

We now describe our results in more detail. In 
\S\ref{section-preliminaries} we recall the original definition:
\begin{defn}[\cite{BondalKapranov-EnhancedTriangulatedCategories}]
\label{defn-intro-bounded-twisted-complexes} 
A \em twisted complex \rm   
over a DG category $\A$ comprises 
\begin{itemize}
\item $\forall\; i \in \mathbb{Z}$, an object $a_i$ of $\A$,
non-zero for only finite number of $i$, 
\item $\forall\; i,j \in \mathbb{Z}$, 
a degree $i - j + 1$ morphism $\alpha_{ij}\colon a_i \rightarrow a_j$ in $\A$,
\end{itemize}
satisfying \begin{equation}
\label{eqn-intro-the-twisted-complex-condition}
(-1)^j d\alpha_{ij} + \sum_k \alpha_{kj} \circ \alpha_{ik} = 0. 
\end{equation}
\end{defn}
The twisted complex condition should be thought of as follows. 
We have Yoneda embedding $\A \hookrightarrow \modA$. Consider 
the object  $\bigoplus a_i[-i]$ in $\modA$. The sum $\sum \alpha_{ij}$ 
is its degree $1$ endomorphism. Let $d_{nat}$ be the natural differential 
on $\bigoplus a_i[-i]$. The condition
\eqref{eqn-intro-the-twisted-complex-condition} is equivalent to 
$d_{nat} + \sum \alpha_{ij}$ squaring to zero. 

In other words, a twisted complex is the data which modifies
$d_{nat}$ to a new differential on $\bigoplus a_i[-i]$. The resulting
new object of $\modA$ is called the \em convolution \rm of the twisted
complex $(a_i, \alpha_{ij})$. In the special case of twisted complexes of 
form $a_0 \rightarrow a_1$ the convolution is simply the cone construction. 

Degree $n$ morphisms $(a_i, \alpha_{ij}) \rightarrow (b_i,
\beta_{ij})$ of twisted complexes are collections $\left\{ f_{ij}
\right\}$ of morphisms $f_{ij}\colon a_i \rightarrow b_j$ in $\A$ 
of degree $n + i - j$. Their composition and differentiation 
are defined so as to ensure that the convolution becomes 
a fully faithful embedding of the resulting DG category $\twcx \A$ into 
$\modA$, see Defn.~\ref{defn-bounded-twisted-complexes}. Indeed, the 
assignment of the module $\bigoplus a_i[-i]$ to a collection of
objects $\left\{ a_i \right\}$ determines the rest of the definitions
of twisted complexes and their morphisms. 

These definitions can be replicated for an infinite
collection $\left\{ a_i \right\}$ resulting in a ``naive'' notion 
of an unbounded twisted complex  $\left\{ a_i, \alpha_{ij} \right\}$, 
see Defn.~\ref{defn-unbounded-twisted-complexes-naive}. 
The sum $\sum \alpha_{ij}$ is now infinite and doesn't necessarily define 
a degree $1$ endomorphism of $\bigoplus a_i[-i]$ in $\modA$, so 
we impose this as an extra condition. It means that only 
a finite number $\alpha_{ij} \neq 0$ for any 
$i \in \mathbb{Z}$, and similarly for the components $f_{ij}$ 
of morphisms of twisted complexes. We again have
$\twcxub_{naive} \A \hookrightarrow \modA$.

However, often the ``naive'' category $\twcxub_{naive} \A$ 
is not what we want. Firstly,  
when $\A$ is not small the category $\modA$ isn't well-defined. The
definition of $\twcxub_{naive} \A$ is still valid, but to have 
the convolution functor we need to enlarge the universe to make $\A$ small. 
More importantly, even small $\A$ might admit countable shifted 
direct sums $\bigoplus a_i[-i]$ of its objects. 

The main subtlety is then that infinite direct sums, unlike
finite, \em do not commute with the Yoneda embedding 
$\A \hookrightarrow \modA$ \rm (Example
\ref{exmpl-yoneda-embedding-doesn't-preserve-direct-sums}). 
The direct sum $\bigoplus a_i[-i]$ 
assigned to a twisted complex $(a_i, \alpha_{ij})$ can thus be taken in $\modA$ 
or in $\A$. The former leads to the ``naive'' category
$\twcxub_{naive} \A$, while the latter to a strictly larger category 
$\twcxub_{\A} \A$ where infinite number of $\alpha_{ij}$ can be
non-zero for any $i \in \mathbb{Z}$ as long as $\sum \alpha_{ij}$ is still an
endomorphism of $\bigoplus a_i[-i]$ in $\A$. The difference
between $\twcxub_{naive} \A$ and $\twcxub_{\A} \A$ lies only
in unbounded twisted complexes. 
If $\A$ admits change of differential (Defn.~\ref{defn-dg-category-admits-change-of-differential}), the convolution functor takes values in $\A$. 
All these considerations apply when $\A = \modC$ for small $\C$
(Example \ref{exmpl-unbounded-twisted-complexes-over-modC}).

This motivates our \S\ref{section-twisted-complexes-unbounded-case} where 
we define unbounded twisted complexes relative to an embedding of $\A$ into
another DG category $\B$:
\begin{defn}[see Definition \ref{defn-unbounded-twisted-complexes-non-naive}] 
\label{defn-intro-unbounded-twisted-complexes-non-naive} 
Let $\A$ be a DG category with a fully faithful embedding into 
a DG category $\B$ which has countable direct sums and shifts. 

An \em unbounded twisted complex \rm over $\A$ relative to $\B$ consists of 
\begin{itemize}
\item $\forall\; i \in \mathbb{Z}$, an object $a_i$ of $\A$,
\item $\forall\; i,j \in \mathbb{Z}$, 
a degree $i - j + 1$ morphism $\alpha_{ij}\colon a_i \rightarrow a_j$ in $\A$,
\end{itemize}
satisfying 
\begin{itemize}
\item $\sum \alpha_{ij}$ is an endomorphism of $\bigoplus_{i \in
\mathbb{Z}} a_i[-i]$ in $\B$, 
\item The twisted complex condition \eqref{eqn-intro-the-twisted-complex-condition}. 
\end{itemize}

The \em DG category $\twcxub_\B(\A)$ of unbounded twisted complexes 
over $\A$ relative to $\B$ \rm is defined in the unique way which
yields fully faithful convolution functor  
\begin{equation}
\label{eqn-intro-convolution-functor-for-twcxub_B-A}
\twcxub_\B(\A) \rightarrow \modB, 
\end{equation}
which sends any $(a_i, \alpha_{ij})$ to $\bigoplus a_i[-i]$ with 
its natural differential modified by $\sum \alpha_{ij}$. 
\end{defn}

Any $\B$ as above admits change of differential 
if and only if it admits convolutions of unbounded twisted complexes.
In such case  for any 
$\A \subseteq \B$ the convolution functor 
\eqref{eqn-intro-convolution-functor-for-twcxub_B-A}
takes values in $\B$ (Lemma
\ref{lemma-admits-convolutions-of-unbounded-twisted-complexes}). We 
can thus use unbounded twisted complexes over non-small categories 
without running into set-theoretic issues.

In \S\ref{section-twisted-bicomplexes} we generalise twisted complexes 
in another direction and define a
\em twisted bicomplex \rm (Defn.~\ref{defn-a-twisted-bicomplex}) over
$\A$. These are bigraded twisted complexes. To work with the unbounded
ones, we again fix an embedding of $\A$ into a DG category $\B$. 
We denote the resulting category of unbounded twisted complexes
by $\twbicx_\B^{\pm}(\A)$. A twisted bicomplex is not a twisted complex 
of its rows or of its columns. It only becomes one after 
a sign twist. We write this down explicitly as a pair of functors 
$$ \cxrow, \cxcol\colon \ 
\twcx_\B^{\pm}\left(\twcx_\B^{\pm}\left(\A\right)\right) 
\rightarrow \twbicx_\B^{\pm}(\A). $$
We relate the images of these functors and show that both become isomorphisms 
if we only work with one-sided twisted complexes and bicomplexes
(Prop.~\ref{prps-functors-cxrow-and-cxcol}). 

Finally, in \S\ref{section-ainfty-structures-in-monoidal-dg-categories}
we give the main application we had in mind: 
to reformulate and generalise the definitions of $\Ainfty$-algebras and modules 
\cite[\S2]{Lefevre-SurLesAInftyCategories} in terms of twisted
complexes. This disposes with the necessity to work explicitly 
with the operation $m_1$ (the differential) and makes the definitions
work in an arbitrary DG monoidal category $\A$ 
(or, more generally, a DG bicategory). 

In \S\ref{section-ainfinity-structures-definitions} we give the
resulting definitions. They all ask for the bar/cobar constructions of 
the $\Ainfty$-operations to be a twisted complex. Since these
constructions involve infinite number of objects, we need 
the theory of unbounded twisted complexes. 
Now, in bar constructions there is only a finite number of 
arrows emerging from each element of the twisted complex. Hence, 
our definitions of an $\Ainfty$-algebra or an $\Ainfty$-module are 
independent of the ambient category $\B$ we use to define unbounded 
twisted complexes. In cobar constructions this is 
no longer the case and the choice of $\B$ matters. 
These definitions are studied further in
\cite{AnnoLogvinenko-AinfinityStructuresInDGMonoidalCategoriesAndStrongHomotopyUnitality}
whose \S3.2 explains at length how they generalise 
the classical definitions \cite[\S2]{Lefevre-SurLesAInftyCategories}
\cite{Lyubashenko-CategoryOfAinftyCategories}. 

In \S\ref{section-twisted-complexes-of-ainfty-modules} 
we look at twisted complexes of $\Ainfty$-modules. 
As per \S\ref{section-ainfinity-structures-definitions} 
let $\nodA$ be the category of $\Ainfty$-modules over an 
$\Ainfty$-algebra $A$ in a monoidal DG category $\A$.  
We define twisted complexes over $\nodA$ neither relative 
to $\modd\text{-}(\nodA)$ nor to $\nodA$. Instead, we embed $\A$
into a cocomplete closed monoidal DG category $\B$ with convolutions
of unbounded twisted complexes and define twisted complexes relative 
to $\nodrepA$, the category of $\Ainfty$-$A$-modules in $\B$. 
As $\nodrepA$ also admits convolutions of unbounded 
twisted complexes
(Cor.~\ref{cor-nod-A^B-has-convolutions-of-unbounded-twisted-complexes}),
convolutions of twisted complexes over $\nodA$ take values in $\nodrepA$. 
We can always set $\B = \modA$ with the induced monoidal structure
\cite[\S4.5]{GyengeKoppensteinerLogvinenko-TheHeisenbergCategoryOfACategory}.
However, we may need to choose differently 
e.g.  for $\Ainfty$-modules in a category of $\Ainfty$-modules. 

We then use the twisted bicomplex techniques we developed in
\S\ref{section-twisted-bicomplexes} to prove that a twisted complex 
of $\Ainfty$-modules defines an $\Ainfty$-module 
structure on the twisted complex of their underlying objects
in a way that gives a fully faithful embedding of the corresponding 
categories (Prps.~\ref{prps-twcxub-nodA-embeds-into-nod-twcxub-A}). 
It follows that the DG category $\nodA$ of $\Ainfty$-modules 
over an $\Ainfty$-algebra $A$ is pretriangulated (resp. admits
convolutions of unbounded twisted complexes) if and only if 
DG monoidal category $\A$ we work is (resp. does)
(Cor.~\ref{cor-noddinfA-strongly-pretriangulated-iff-A-is}). 
 
In the Appendix we describe a homotopy transfer of structure for 
$\Ainfty$-modules. 

We are aware of an alternative definition of the DG category 
of unbounded twisted complexes in
\cite{Genovese-TStructuresOnUnboundedTwistedComplexes}.  
It ignores the subtleties we consider by imposing no finiteness 
conditions on the differentials $\alpha_{ij}$ in twisted complexes and the
components $f_{ij}$ of their morphisms. The resulting category admits 
no convolution functor and is better suited to purposes different from
ours. 

\em Acknowledgements: \rm We would like to thank Sergey Arkhipov,
Alexander Efimov, and Dmitri Kaledin for useful discussions. 
The first author would like to thank Kansas State
University for providing a stimulating research environment while
working on this paper. The second author would like to offer similar
thanks to Cardiff University and to Max-Planck-Institut f{\"u}r
Mathematik Bonn. 

\section{Preliminaries}
\label{section-preliminaries}

\subsection{DG categories}
\label{section-dg-categories}

For a brief introduction to DG-categories, DG-modules and 
the technical notions involved we direct the reader to 
a survey in \cite{AnnoLogvinenko-SphericalDGFunctors}, \S2-4. 
Other nice sources are \cite{Keller-DerivingDGCategories},
\cite{Toen-TheHomotopyTheoryOfDGCategoriesAndDerivedMoritaTheory},
\cite{Toen-LecturesOnDGCategories}, and 
\cite{LuntsOrlov-UniquenessOfEnhancementForTriangulatedCategories}. 

We summarise the key notions relevant to this paper. Throughout 
the paper we work in a fixed universe $\mathbb{U}$ of sets 
containing an infinite set. We also fix the base field or commutative 
ring $k$ we work over. 

We define $\modk$ to be the category of $\mathbb{U}$-small complexes of 
$k$-modules. It is a cocomplete closed symmetric monoidal category
with monoidal operation $\otimes_k$ and unit $k$. A \em DG category \rm 
is a category enriched over $\modk$. In particular, any DG category 
is locally small. 

If a DG category $\A$ is small, we write $\modA$ for the DG category
of (right) $\A$-modules. These are functors $\Aopp \rightarrow \modk$, 
so $\modA = \DGFun(\Aopp,\modk)$. Note that if $\A$ is not small, then 
$\modA$ doesn't make sense. It isn't even a DG category in the above sense - 
its morphism spaces are no longer small and hence do not lie in $\modk$. 

We can always enlarge our universe $\mathbb{U}$ to a universe $\mathbb{V}$
where $\A$ is small. This enlarges $\modk$ and hence $\modA$ depends
on choice of $\mathbb{V}$. However, in this paper we only work with $\modA$ 
as a target for the convolution of twisted complexes over $\A$. For
these purposes, the choice of $\mathbb{V}$ doesn't matter - the only part
of $\modA$ we interact with are countable direct sums of shifts of
objects of $\A$ with modified differential. 

Thus, when $\A$ is not small, we mean by $\modA$ the module category
of $\A$ taken in any appropriate enlargement $\mathbb{V}$ of
$\mathbb{U}$. Moreover, the constructions in this paper, such as
that of the category of twisted complexes over $\A$ taken 
relative to a DG category $\B$, were devised precisely to enable 
us to replace $\modA$ with something more approriate when $\A$ is not
small. 

\subsection{Key isomorphism} 
\label{section-key-isomorphism-for-twisted-complexes}
Let $\A$ be a DG-category, let $E,F \in \modA$ and 
$i,j \in \mathbb{Z}$. The theory of twisted complexes
\cite{BondalKapranov-EnhancedTriangulatedCategories} which we
summarise in \S\ref{section-twisted-complexes-summary} depends 
crucially on the choice of an isomorphism  
\begin{equation}
\label{eqn-key-isomorphism-for-twisted-complexes}
\homm_\A(E,F)[j-i] \xrightarrow{\quad\sim\quad} \homm_\A(E[i],F[j]). 
\end{equation}

The simplest such isomorphism is:
\begin{defn}
Let $\A$ be a DG-category, let $E,F \in \modA$ and 
$i,j \in \mathbb{Z}$. Define the isomorphism of graded $k$-modules
\begin{equation}
\psi\colon \homm_\A(E,F)[j-i] \xrightarrow{\quad\sim\quad} \homm_\A(E[i],F[j]) 
\end{equation}
to be the map which sends any $f \in
\homm^p_\A(E,F)$ to itself considered as an element of
$\homm^{p-j+i}(E[i], F[j])$. In other words, forgetting the grading,
in every fiber over every $a \in \A$ the map $\psi(f)$ is the same map of
$k$-vector spaces as $f$. 
\end{defn}

Note that $\psi$ is \em not \rm compatible with the differentials:
$$ d_{\homm_\A(E[i],F[j])} \circ \psi = (-1)^i \psi \circ
d_{\homm_\A(E,F)[j-i]}. $$

There are at least two natural ways to fix this. Define
$$ \psi_1, \psi_2\colon \homm_\A(E,F)[j-i]
\xrightarrow{\quad\sim\quad} \homm_\A(E[i],F[j]) $$
to be the maps which send $f \in \homm^p_\A(E,F)$
to $(-1)^{ip} \psi(f)$ and $(-1)^{i(p - j + i)}\psi(f)$. The difference
between the two lies in whether we multiply $i$ by the degree of $f$
in $\homm_\A(E,F)$ or its degree in $\homm_\A(E,F)[j-i]$. 

Both $\psi_1$ and $\psi_2$ are isomorphisms of DG $k$-modules. 
However, they are incompatible with the composition. By this 
we mean the following: let $E,F,G \in \modA$ and $i,j,k \in
\mathbb{Z}$, then e.g. the isomorphism 
$$ \psi_1(E,i,G,k)\colon \homm_\A(E,G)[k-i]
\xrightarrow{\quad\sim\quad} \homm_\A(E[i],G[k]) $$
is not a composition of $\psi_1(E,i,F,j)$ and $\psi_1(F,j,G,k)$. 
On the other hand, $\psi$, while incompatible with
differentials, is compatible with composition. 

The theory of twisted complexes and its fundamental definitions depend
on the choice of an isomorphism 
\eqref{eqn-key-isomorphism-for-twisted-complexes}. The definition 
of the DG category $\twcx(\A)$ of twisted complexes over $\A$ 
is set up so as to ensure that there exists a fully faithful functor
$\twcx(\A) \hookrightarrow \modA$ called \em convolution\rm, 
cf.~\S\ref{section-twisted-complexes-summary}. This functor is defined 
using the isomorphism $\eqref{eqn-key-isomorphism-for-twisted-complexes}$, 
thus different choices would lead to different
formulas in the definition of $\twcx(\A)$. 

The incompatibility of $\psi$ with differentials introduces
in these formulas a simple sign to every appearance
of the differential $d_\A$ of $\A$, 
cf.~\eqref{eqn-the-twisted-complex-condition} and 
\eqref{eqn-the-hom-complex-of-twisted-complexes-differential}. 
On the other hand, the incompatibility of $\psi_1$ and $\psi_2$ 
with composition introduces into the same formulas 
a complicated sign to every composition of two
morphisms of $\A$.

We choose to use the graded module isomorphim $\psi$
to identify $\homm_\A(E,F)[j-i]$ with $\homm_\A(E[i],F[j])$ when 
defining twisted complexes. We fix this choice and use it implicitly
in the sections below.  

\subsection{Bounded twisted complexes}
\label{section-twisted-complexes-summary}

Here we summarise some known facts about the usual, bounded twisted 
complexes. This notion was originally introduced by Bondal and 
Kapranov in \cite{BondalKapranov-EnhancedTriangulatedCategories}:

\begin{defn}
\label{defn-bounded-twisted-complexes} 
A \em twisted complex \rm   
over a DG category $\A$ is a collection of 
\begin{itemize}
\item $\forall\; i \in \mathbb{Z}$, an object $a_i$ of $\A$,
non-zero for only finite number of $i$, 
\item $\forall\; i,j \in \mathbb{Z}$, 
a degree $i - j + 1$ morphism $\alpha_{ij}\colon a_i \rightarrow a_j$ in $\A$,
\end{itemize}
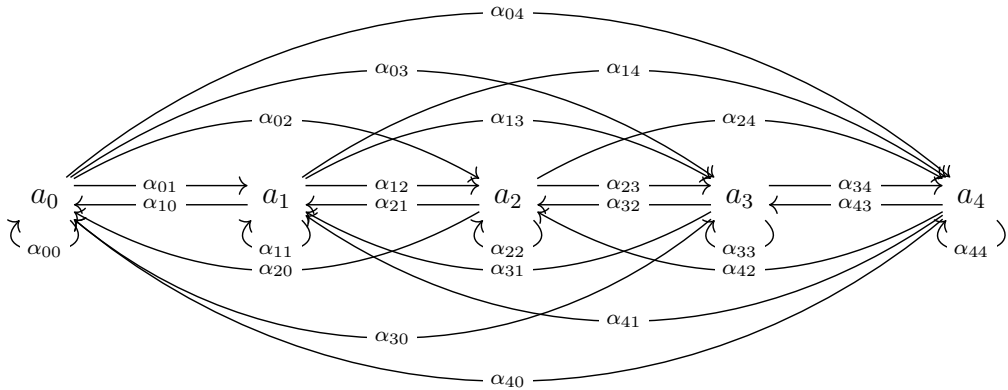
\begin{equation*}
\begin{tikzcd}[column sep = 2.3cm, ampersand replacement=\&]
a_0
\ar[loop, out=-45, in=-135, looseness=3]{}[description]{\alpha_{00}}
\ar[shift left=1.25]{r}[description]{\alpha_{01}}
\ar[bend left=30]{rr}[description]{\alpha_{02}}
\ar[bend left=35]{rrr}[description]{\alpha_{03}}
\ar[bend left=40]{rrrr}[description]{\alpha_{04}}
\&
a_1
\ar[loop, out=-45, in=-135, looseness=3]{}[description]{\alpha_{11}}
\ar[shift left=1.25]{l}[description]{\alpha_{10}}
\ar[shift left=1.25]{r}[description]{\alpha_{12}}
\ar[bend left=30]{rr}[description]{\alpha_{13}}
\ar[bend left=35]{rrr}[description]{\alpha_{14}}
\&
a_2
\ar[loop, out=-45, in=-135, looseness=3]{}[description]{\alpha_{22}}
\ar[bend left=30]{ll}[description]{\alpha_{20}}
\ar[shift left=1.25]{l}[description]{\alpha_{21}}
\ar[shift left=1.25]{r}[description]{\alpha_{23}}
\ar[bend left=30]{rr}[description]{\alpha_{24}}
\&
a_3
\ar[loop, out=-45, in=-135, looseness=3]{}[description]{\alpha_{33}}
\ar[bend left=40]{lll}[description]{\alpha_{30}}
\ar[bend left=30]{ll}[description]{\alpha_{31}}
\ar[shift left=1.25]{l}[description]{\alpha_{32}}
\ar[shift left=1.25]{r}[description]{\alpha_{34}}
\&
a_4
\ar[loop, out=-45, in=-135, looseness=3]{}[description]{\alpha_{44}}
\ar[bend left=40]{llll}[description]{\alpha_{40}}
\ar[bend left=35]{lll}[description]{\alpha_{41}}
\ar[bend left=30]{ll}[description]{\alpha_{42}}
\ar[shift left=1.25]{l}[description]{\alpha_{43}}
\end{tikzcd}
\end{equation*}
satisfying the condition
\begin{equation}
\label{eqn-the-twisted-complex-condition}
(-1)^j d\alpha_{ij} + \sum_k \alpha_{kj} \circ \alpha_{ik} = 0. 
\end{equation}

Define the \em DG category $\twcx(\A)$ of twisted complexes over $\A$ \rm  
by setting
\begin{equation}
\label{eqn-the-hom-complex-of-twisted-complexes}
\homm^\bullet_{\twcx(\A)}\bigl((a_i, \alpha_{ij}),(b_i,
\beta_{ij})\bigr) := \bigoplus_{q,k,l \in \mathbb{Z}} \homm^q_{\A}(a_k, b_l)
\end{equation}
where each  $f \in  \homm^q_{\A}(a_k, b_l)$ has degree $q + l - k$ and 
\begin{equation}
\label{eqn-the-hom-complex-of-twisted-complexes-differential}
 df := (-1)^l d_{\A} f + \sum_{m \in
\mathbb{Z}}\left( \beta_{lm} \circ f - (-1)^{q + l -k} f 
\circ \alpha_{mk} \right), 
\end{equation}
where $d_\A$ is the differential on morphisms in $\A$.
\end{defn}

This definition ensures that $\twcx(\A)$ is 
isomorphic to the full subcategory of $\modA$ consisting of 
the DG $\A$-modules whose underlying graded modules are of form 
$\oplus_{i \in \mathbb{Z}} a_i[-i]$ with only finite number
of $a_i \in \A$ non-zero. Indeed:
\begin{itemize}
\item The twisted complex condition \eqref{eqn-the-twisted-complex-condition}
is equivalent to $d_{nat} + \sum_{i,j} \alpha_{ij}$ being another
differential on $\oplus_{i \in \mathbb{Z}} a_i[-i]$. Here $d_{nat}$ is
its natural differential. 
\item The $\homm$-complex 
\eqref{eqn-the-hom-complex-of-twisted-complexes-differential}
is defined to have the same underlying graded $k$-module as
\begin{equation}
\label{eqn-the-hom-complex-of-convolutions-in-mod-A}
\homm^\bullet_{\modA}\left(\bigoplus_{k \in \mathbb{Z}}
a_k[-k], \bigoplus_{l \in \mathbb{Z}} b_l[-l]\right)
= 
\bigoplus_{k,l \in \mathbb{Z}} \homm^{\bullet - l + k}_{\A} 
\left(a_k, b_l\right).
\end{equation}
and the differential
$\eqref{eqn-the-hom-complex-of-twisted-complexes-differential}$ 
is defined so as to coincide under this identification
with the one obtained on 
\eqref{eqn-the-hom-complex-of-convolutions-in-mod-A}
by endowing the two direct sums with their new differentials. 
\end{itemize}

We thus have a fully faithful \em convolution functor \rm 
$$ \conv\colon\; \twcx(\A) \hookrightarrow \modA $$
which sends each $(a_i, \alpha_{ij})$ to the $\A$-module $\bigoplus a_i[-i]$
equipped with the new differential $d_{nat} + \sum \alpha_{ij}$. Note, that 
the existence of this functor can be used as the definition of 
the category $\twcx(\A)$ once one fixes the assignment of the graded 
module $\bigoplus a_i[-i]$ to any collection $\left\{ a_i \right\}_{i
\in \mathbb{Z}}$. 

A twisted complex is called \em one-sided \rm if $\alpha_{ij} = 0$ for 
all $i \geq j$. 
\begin{equation}
\begin{tikzcd}[column sep = 2.4cm, ampersand replacement=\&]
a_0
\ar{r}[description]{\alpha_{01}}
\ar[bend left=20]{rr}[description]{\alpha_{02}}
\ar[bend left=25]{rrr}[description]{\alpha_{03}}
\ar[bend left=30]{rrrr}[description]{\alpha_{04}}
\&
a_1
\ar{r}[description]{\alpha_{12}}
\ar[bend left=20]{rr}[description]{\alpha_{13}}
\ar[bend left=25]{rrr}[description]{\alpha_{14}}
\&
a_2
\ar{r}[description]{\alpha_{23}}
\ar[bend left=20]{rr}[description]{\alpha_{24}}
\&
a_3
\ar{r}[description]{\alpha_{34}}
\&
a_4
\end{tikzcd}
\end{equation}
If $(a_i, \alpha_{ij})$ is a one-sided twisted complex over $\A$, then $(a_i,
\alpha_{ij})$ is a (usual) complex over $H^0(\A)$. Thus one-sided
twisted complexes can be considered as homotopy lifts to $\A$ of usual 
complexes in $H^0(\A)$. The full subcategory of $\twcx(\A)$ consisting
of one-sided twisted complexes is called the \em pretriangulated
hull \rm of $\A$ and is denoted $\pretriag(\A)$. We say that a DG
category is \em pretriangulated \rm (resp. \em strongly
pretriangulated\rm) if the natural embedding $\A \hookrightarrow
\pretriag(\A)$ is a quasi-equivalence (resp. equivalence). 

The reason for the term ``pretriangulated'' is that 
$H^0(\pretriag(\A))$ is the triangulated hull of $H^0(\A)$ in $H^0(\modA)$, 
or indeed any $H^0(\B)$ for any fully faithful embedding of $\A$ into a
pretriangulated DG category $\B$. 

\section{Unbounded twisted complexes}
\label{section-twisted-complexes-unbounded-case}

In this section we generalise the notions in 
\S\ref{section-twisted-complexes-summary} to
unbounded twisted complexes. The generalisation
seems straightforward, but there are subtleties 
regarding infinite direct sums. Unlike finite direct sums, 
these are not preserved by all DG-functors. In particular, 
\underline{they are not preserved by the Yoneda embedding}
$$ \Upsilon\colon \A \hookrightarrow \modA, \quad \quad a \mapsto
\homm_A(-,a) \quad \forall\; a \in \A $$ 
which we used implicitly in defining the convolution of a twisted
complex. 

\begin{exmpl}
\label{exmpl-yoneda-embedding-doesn't-preserve-direct-sums}
Let $\left\{ a_i \right\}_{i \in \mathbb{Z}}$ be objects in $\A$
such that $\bigoplus_{i \in \mathbb{Z}} a_i$ exists in $\A$. Then 
\begin{align*}
\bigoplus_{i \in \mathbb{Z}} \Upsilon(a_i) &= \bigoplus_{i \in \mathbb{Z}}
\homm_{\A}(-,a_i), \\
\Upsilon\left(\bigoplus_{i \in \mathbb{Z}} a_i\right)  &= 
\homm_{\A}(-,\bigoplus_{i \in \mathbb{Z}} a_i), 
\end{align*}
are two different $\A$-modules, with the former being a strict submodule 
of the latter. Let $b \in \A$, the morphisms from $\Upsilon(b)$ 
to the former module are the finite sums of $b \rightarrow a_i$.  
On the other hand, the morphisms from $\Upsilon(b)$ to the latter 
are the morphisms from $b$ to $\bigoplus_{i \in \mathbb{Z}}
a_i$, which includes some infinite sums of $b \rightarrow a_i$. 
In particular, if $b = \bigoplus_{i \in \mathbb{Z}} a_i$, 
then $\id_b$ is the infinite sum of $\id_{a_i}$. 
\end{exmpl}

To define an unbounded twisted complex of objects 
$\left\{ a_i \right\}_{i \in \mathbb{Z}}$ of $\A$, we need  
to choose in which category we take the infinite direct sum 
$\bigoplus_{i \in \mathbb{Z}} a_i[-i]$. We can always do it in $\modA$.
Then, proceeding as before, we arrive at the following definition. In it, we 
allow infinite number of non-zero objects $a_i$, but then,  
both for twisted differentials and for the morphisms of twisted complexes,
we disallow an infinite number of non-zero maps to emerge from any 
one object $a_i$:

\begin{defn}[``Naive'' version]
\label{defn-unbounded-twisted-complexes-naive} 

An \em unbounded twisted complex \rm   
over a DG category $\A$ consists of 
\begin{itemize}
\item $\forall\; i \in \mathbb{Z}$, an object $a_i$ of $\A$,
\item $\forall\; i,j \in \mathbb{Z}$, 
a degree $i - j + 1$ morphism $\alpha_{ij}\colon a_i \rightarrow a_j$ in $\A$,
\end{itemize}
satisfying 
\begin{itemize}
\item For any $i \in \mathbb{Z}$ only finite number of $\alpha_{ij}$ 
are non-zero, 
\item The twisted complex condition \eqref{eqn-the-twisted-complex-condition}. 
\end{itemize}

Define \em DG category $\twcxub_{naive}(\A)$ of unbounded twisted complexes 
over $\A$ \rm by 
\begin{equation}
\label{eqn-the-hom-complex-of-unbounded-twisted-complexes-naive}
\homm^\bullet_{\twcxub_{naive}(\A)}\bigl((a_i, \alpha_{ij}),(b_i,
\beta_{ij})\bigr) := \bigoplus_{q \in \mathbb{Z}} \prod_{k \in
\mathbb{Z}}\bigoplus_{l \in \mathbb{Z}} \homm^q_{\A}(a_k, b_l)
\end{equation}
where the degree of $\homm^q_{\A}(a_k, b_l)$ is $q + l - k$ and 
the differential is defined by
\eqref{eqn-the-hom-complex-of-twisted-complexes-differential}.
\end{defn}

As before, this results in the fully faithful convolution functor
$$ \conv \colon \twcxub_{naive}(\A) \hookrightarrow \modA. $$ 
Apriori, this is the only definition we can make for an arbitrary
DG category $\A$. Indeed, unless specifically mentioned otherwise, 
we write $\twcxub(\A)$ for $\twcxub_{naive}(\A)$. 

However, in some cases it is useful to define $\twcxub(\A)$ to be
bigger than $\twcxub_{naive}(\A)$: 
\begin{exmpl}
\label{exmpl-unbounded-twisted-complexes-over-modC}
Let $\A = \modC$ for some small DG category $\C$. Assign to a collection 
$\left\{ a_i \right\}_{i \in \mathbb{Z}}$ the representable 
$\A$-module $\Upsilon(\bigoplus_{i \in \mathbb{Z}} a_i[-i])$,
instead of the non-representable $\A$-module 
$\bigoplus_{i \in \mathbb{Z}} \Upsilon(a_i[-i])$. 
This yields the definition of $\twcxub(\A)$ which is analogous 
to the naive one above, except we do allow infinite number of 
twisted differentials $\alpha_{ij}$
to emerge from a single object $a_i$ as long as $\sum \alpha_{ij}$
defines an endomorphism of $\bigoplus_{i \in \mathbb{Z}} a_i[-i]$ in $\A$, 
and similarly for morphisms of twisted complexes. 
As before, this definition ensures that we have 
the fully faithful convolution functor $\twcxub(\A) \hookrightarrow \modA$.  
However, since $\A = \modC$ is closed under the change of differential,
this convolution filters through the Yoneda embedding. Thus we have 
the fully faithful functor
$$\conv\colon\; \twcxub(\A) \rightarrow \A. $$
In fact, it is an equivalence, since it has a right
inverse - the tautological embedding $\A \hookrightarrow \twcxub(\A)$
which sends any $a \in \A$ to itself considered as a trivial twisted
complex concentrated in degree zero. We thus see that $\A = \modC$ is
closed under convolutions of all unbounded twisted complexes in
$\twcxub(\A)$.  
\end{exmpl}

Finally, even when $\A$ does not admit all small direct sums, 
there might still be a better category to take these in than $\modA$. 
For example, $\A$ might be a full subcategory of some $\modC$ containing 
some infinite direct sums, but not all of them. Another example, 
which indeed motivated these considerations, can be found in 
\S\ref{section-twisted-complexes-of-ainfty-modules}. 
We thus define the following:
\begin{defn}
\label{defn-unbounded-twisted-complexes-non-naive} 
Let $\A$ be a DG category with a fully faithful embedding into 
a DG category $\B$ which has countable direct sums and shifts. 

An \em unbounded twisted complex \rm over $\A$ relative to $\B$ consists of 
\begin{itemize}
\item $\forall\; i \in \mathbb{Z}$, an object $a_i$ of $\A$,
\item $\forall\; i,j \in \mathbb{Z}$, 
a degree $i - j + 1$ morphism $\alpha_{ij}\colon a_i \rightarrow a_j$ in $\A$,
\end{itemize}
satisfying 
\begin{itemize}
\item $\sum \alpha_{ij}$ is an endomorphism of $\bigoplus_{i \in
\mathbb{Z}} a_i[-i]$ in $\B$, 
\item The twisted complex condition \eqref{eqn-the-twisted-complex-condition}. 
\end{itemize}

Define \em DG category $\twcx^{\pm}_\B(\A)$ of unbounded twisted complexes 
over $\A$ relative to $\B$ \rm by setting
\begin{equation}
\label{eqn-the-hom-complex-of-unbounded-twisted-complexes-non-naive}
\homm^\bullet_{\twcx^{\pm}_{\B}(\A)}\bigl((a_i, \alpha_{ij}),(b_i,
\beta_{ij})\bigr) := \homm^\bullet_{\B}(\bigoplus_{k \in \mathbb{Z}} a_k[-k],
 \bigoplus_{l \in \mathbb{Z}} b_l[-l]) 
\end{equation}
wih its natural grading and the differential defined by
\eqref{eqn-the-hom-complex-of-twisted-complexes-differential}.
\end{defn}

Where the choice of $\B$ is clear or was fixed, we shall write 
$\twcxub(\A)$ for $\twcx^{\pm}_\B(\A)$. 

As before, our definition ensures that we have a fully faithful
convolution functor 
$$ \conv\colon\; \twcx^{\pm}_{\B}(\A) \rightarrow \modB. $$
We have the commutative square of fully faithful embeddings
\begin{equation}
\begin{tikzcd}
\A 
\ar[hookrightarrow]{r}{I} 
\ar{d}[hookrightarrow]{\Upsilon}
&
\B 
\ar{d}[hookrightarrow]{\Upsilon}
\\
\modA
\ar[hookrightarrow]{r}{I^*}
&
\modB.
\end{tikzcd}
\end{equation}
Since all DG-functors preserve finite direct sums, we see that
on bounded twisted complexes the convolution functor into $\modB$
is simply the composition of the usual convolution functor into
$\modA$ and $I^*$.

Observer that setting $\B = \modA$ recovers the definition of 
$\twcx^{\pm}_{naive}(\A)$ with the convolution into $\modA$. 
On the other hand, when we have $\A = \modC$ for some small DG-category $\C$, 
setting $\B = \A$ recovers the category constructed in Example
\ref{exmpl-unbounded-twisted-complexes-over-modC} with its
convolution into $\modA$ which filters through $\A$. 

\begin{defn}
\label{defn-dg-category-admits-change-of-differential}
A DG category $\B$ \em admits 
change of differential \rm if for all $b \in \B$ and $f \in
\homm^1_\B(b,b)$ with $df + f^2 = 0$ the module in $\modB$
which has the underlying graded module of $\homm_\B(-,b)$
and the differential $d_{\homm_\B(-,b)} + f$ is representable. 
\end{defn}
\begin{defn}
A DG category $\B$ \em admits
convolutions of unbounded twisted complexes \rm if it
admits countable direct sums and shifts and the convolution 
functor $\twcx^{\pm}_\B(\B) \hookrightarrow \modB$ filters 
through $\B \hookrightarrow \modB$. 
\end{defn}

We do not need to specify which unbounded twisted complexes
$\B$ admits convolutions of, because for the convolution to be
representable the infinite direct sum needs to be taken in $\B$ itself. 
Thus we need to consider unbounded twisted complexes relative to $\B$
itself. 

If $\B$ admits convolutions of unbounded twisted complexes,
then the convolution  
$$ \twcxub_\B(\B) \hookrightarrow \B $$
is necessarily an equivalence. It is fully faithful 
and has a right inverse which sends any $b \in \B$ to itself
considered as trivial twisted complex in degree zero. 

\begin{lemma}
\label{lemma-admits-convolutions-of-unbounded-twisted-complexes}
Let $\B$ be a DG-category which admits countable direct sums and
shifts. The following are equivalent:
\begin{enumerate}
\item 
\label{item-B-admits-countable-direct-sums-shifts-and-change-of-differential}
$\B$ admits change of differential. 
\item 
\label{item-B-admits-convolutions-of-unbounded-twisted-complexes}
$\B$ admits convolutions of unbounded twisted complexes. 
\item 
\label{item-B-to-tw_B_B-is-equivalence}
The embedding $\B \hookrightarrow \twcxub_\B(\B)$
which sends any $b \in \B$ to itself
considered as a trivial twisted complex in degree zero is an
equivalence. 
\item 
\label{item-for-any-A-tw_B-A-convolves-to-B}
For any DG-category $\A$ with an embedding into $\B$, 
$\twcx^{\pm}_\B(\A) \rightarrow \modB$ filters through 
$\B \hookrightarrow \modB$.
\end{enumerate}
\end{lemma}

\begin{proof}
\underline{
$\eqref{item-B-admits-countable-direct-sums-shifts-and-change-of-differential}
\Rightarrow
\eqref{item-B-admits-convolutions-of-unbounded-twisted-complexes}$}:
This is the same argument as in 
Example \ref{exmpl-unbounded-twisted-complexes-over-modC}. 

\underline{$\eqref{item-B-admits-convolutions-of-unbounded-twisted-complexes}
\Leftrightarrow 
\eqref{item-B-to-tw_B_B-is-equivalence}$}:
The composition 
$$ \B \hookrightarrow \twcxub_\B(\B) \hookrightarrow \modB $$
is the Yoneda embedding. Thus $\twcxub_\B(\B) \hookrightarrow \modB$
filters through the Yoneda embedding if and only if 
$\B \hookrightarrow \twcxub_\B(\B)$ admits a right quasi-inverse. A fully
faithful functor admits a right quasi-inverse if and only if it is an
equivalence. 

\underline{
$\eqref{item-B-admits-convolutions-of-unbounded-twisted-complexes}
\Leftrightarrow 
\eqref{item-for-any-A-tw_B-A-convolves-to-B}$}:

The ``if'' implication is obvious. The ``only if' 
one results from the following commutative
triangle of fully faithful functors: 
\begin{equation*}
\begin{tikzcd}
\twcxub_\B(\A)
\ar[hookrightarrow]{d}
\ar[hookrightarrow]{rr}{\conv}
& &
\modB.
\\
\twcxub_\B(\B)
\ar[hookrightarrow]{urr}[']{\conv}
& &
\end{tikzcd}
\end{equation*}

\underline{
$\eqref{item-B-admits-convolutions-of-unbounded-twisted-complexes}
\Rightarrow 
\eqref{item-B-admits-countable-direct-sums-shifts-and-change-of-differential}$}:

Let $b \in \B$ and $f \in \homm_\B^{1}(b,b)$ with $df + f^2 = 0$. 
Then the complex consisting of $b$ in degree $0$ with a single differential $f$ 
from $b$ to itself is a twisted complex. Its convolution in $\modB$
has the same graded module as $b$ and the differential $d_b + f$.
Since $\B$ admits convolutions of twisted complexes, it is representable. 
\end{proof}

Finally, for any version of $\twcxub(\A)$, we define $\twcxpls(\A)$
and $\twcxmns(\A)$ to be its full subcategories consisting of all
bounded above twisted complexes and all bounded below twisted
complexes, respectively. We also define $\pretriagub(\A)$,
$\pretriagpls(\A)$, and $\pretriagmns(\A)$ to be the full
subcategories of $\twcxub(\A)$, $\twcxpls(\A)$, and $\twcxmns(\A)$
consisting of one-sided twisted complexes. 

\section{Twisted bicomplexes}
\label{section-twisted-bicomplexes}

The following is a natural generalisation of the notion of a twisted
complex:
\begin{defn}
\label{defn-a-twisted-bicomplex}
A \em twisted bicomplex \rm $(a_{ij}, \alpha_{ijkl})$ over a DG category $\A$ 
comprises 
\begin{itemize}
\item $\forall\; i,j \in \mathbb{Z}$, an object $a_{ij}$ of $\A$,
non-zero for only finite number of pairs $(i,j)$, 
\item $\forall\; i,j,k,l \in \mathbb{Z}$, 
a degree $(i+j) - (j+k) + 1$ morphism 
$\alpha_{ijkl}\colon a_{ij} \rightarrow a_{kl}$ in $\A$,
\end{itemize}
satisfying
\begin{equation}
\label{eqn-the-twisted-bicomplex-condition}
(-1)^{k+l} d\alpha_{ijkl} + \sum_{m,n} \alpha_{ijmn} \circ \alpha_{mnkl} = 0. 
\end{equation}

Define the \em DG category $\twbicx(\A)$ of twisted bicomplexes over $\A$ \rm  
by setting
\begin{equation}
\label{eqn-the-hom-complex-of-twisted-bicomplexes}
\homm^\bullet_{\twcx(\A)}\bigl((a_{ij}, \alpha_{ijkl}),(b_{ij},
\beta_{ijkl})\bigr) := \bigoplus_{q,k,l,m,n \in \mathbb{Z}} 
\homm^q_{\A}(a_{kl}, b_{mn})
\end{equation}
where each  $f \in  \homm^q_{\A}(a_{kl}, b_{mn})$ has degree $q +
(m+n) - (k+l)$ and 
\begin{equation}
\label{eqn-the-hom-complex-of-twisted-bicomplexes-differential}
 df := (-1)^{m+n} d_{\A} f + \sum_{p,q \in
\mathbb{Z}}\left( \beta_{mnpq} \circ f - (-1)^{q + (m+n) - (k+l)} f 
\circ \alpha_{pqkl} \right), 
\end{equation}
where $d_\A$ is the differential on morphisms in $\A$.

We think of indices $i$ and $j$ of each $a_{ij}$ as the row index and
the column index, respectively. 
We say that a twisted bicomplex is \em horizontally one-sided \rm
(resp.~\em vertically one-sided\rm) if $\alpha_{ijkl} = 0$ when $l \leq j$ 
(resp.~$k \leq i$). We say that a twisted bicomplex is \em
one-sided \rm if it is both vertically and horizontally one-sided. 
\end{defn}

The ``naive'' categories $\twbicxub(\A)$, $\twbicxpls(\A)$ and $\twbicxmns(\A)$
of unbounded, unbounded below, and unbounded above twisted bicomplexes,
as well as their versions relative to another DG category $\B$ are
defined similarly to the way they are defined in 
\S\ref{section-twisted-complexes-unbounded-case} for twisted complexes.  
We also use $\twbicx^{vos}$, $\twbicx^{hos}$ and $\twbios$ 
to denote the full subcategories consisting of one-sided twisted bicomplexes. 

Finally, we note that the category of twisted
bicomplexes is naturally isomorphic to the category of twisted
complexes of twisted complexes, but in two different ways:
the complex of sign-twisted rows and the complex of sign-twisted
columns. For this result, the twisted complexes over $\A$ 
need to be considered relative to some $\B$ 
which admits the convolutions of unbounded twisted complexes, cf. Lemma 
\ref{lemma-admits-convolutions-of-unbounded-twisted-complexes}. 
This is always true when $\B = \modC$ for some DG-category $\C$. 

\begin{defn}
\label{defn-the-functor-of-rows-and-the-functor-of-columns}
Let $\A$ be DG category and fix its embedding into a DG-category $\B$ 
which admits convolutions of unbounded twisted complexes. Define 
\begin{equation}
\cxrow\colon 
\twcx_\B^{\pm}\left(\twcx_\B^{\pm}\left(\A\right)\right) 
\rightarrow 
\twbicx_\B^{\pm}(\A), 
\end{equation}
\begin{equation}
\cxcol\colon 
\twcx_\B^{\pm}\left(\twcx_\B^{\pm}\left(\A\right)\right) 
\rightarrow 
\twbicx_\B^{\pm}(\A), 
\end{equation}
as follows. Let $(E_i, \alpha_{ik})$ be an object of  
$\twcx_\B^{\pm}\left(\twcx_\B^{\pm}\left(\A\right)\right)$. Write 
\begin{itemize}
\item $E_{i,j}$ for the objects of each $E_i$ and
$\alpha_{i,jl}: E_{i,j} \rightarrow E_{i,l}$ for its differentials. 
\item $\alpha_{ik,jl}: E_{i,j} \rightarrow E_{k,l}$ for the components 
of $\alpha_{ik}\colon E_i \rightarrow E_k$. 
\end{itemize} 
We then define: 
\begin{itemize}
\item $\cxrow(E_i, \alpha_{ik})$ to be the twisted bicomplex
whose $ij$-th object is $E_{i,j}$ and whose $ijkl$-th differential 
is $\alpha_{ik,jl}$ if $i \neq k$ and 
$\alpha_{ii,jl} + (-1)^i \alpha_{i,jl}$ if $i = k$. 
\item $\cxcol(E_i, \alpha_{ik})$ to be the twisted bicomplex
whose $ij$-th object is $E_{j,i}$ and whose $ijkl$-th differential 
is $\alpha_{jl,ik}$ if $j \neq l$ and $\alpha_{jj,ik} +
(-1)^j \alpha_{j,ik}$ if $j = l$. 
\end{itemize}

Similarly, let $f: (E_i, \alpha_{ik}) \rightarrow (F_i, \beta_{ik})$
be a morphism in $\twcx_\B^{\pm}\left(\twcx_\B^{\pm}\left(\A\right)\right)$. 
Write $f_{ik}$ for its $E_i \rightarrow F_k$ component, and then
$f_{ik,jl}$ for $E_{i,j} \rightarrow F_{k,l}$ component of that. We
then define:
\begin{itemize}
\item $\cxrow(f)$ to be the bicomplex map whose $ijkl$-th component is
$f_{ik,jl}$,
\item $\cxcol(f)$ to be the bicomplex map whose $ijkl$-th component is
$f_{jl,ik}$. 
\end{itemize}
\end{defn}

In other words, $\cxrow(E_i, \alpha_{ik})$ is
the bicomplex whose $i$-th row is the twisted complex $(-1)^i E_i$ 
to whose differentials we further add all the components of
$\alpha_{ii}$. The differentials between $i$-th and $j$-th rows for
$i \neq j$ are the components of $\alpha_{ij}$. Similarly, 
$\cxcol(E_i, \alpha_{ik})$ is the bicomplex whose columns are $(-1)^i E_i$ 
modified by $\alpha_{ii}$ and whose intercolumn differentials are
$\alpha_{ij}$. On morphisms, both functors simply map a morphism 
of complexes of complexes to the bicomplex morphism with the
same components. 

By our assumption on $\B$, the convolution functor embeds
$\twcx_\B^{\pm}\left(\A\right)$ fully faithfully into $\B$. 
We thus have a double convolution functor:
\begin{equation}
\label{eqn-double-convolution-of-complexes-of-complexes} 
\conv\conv\colon
\twcx_\B^{\pm}\left(\twcx_\B^{\pm}\left(\A\right)\right)
\hookrightarrow
\B. 
\end{equation}

\begin{prps}
\label{prps-functors-cxrow-and-cxcol}
Let $\A$ be a DG category with a fully faithful functor into a DG category $\B$ 
which admits convolutions of unbounded twisted complexes. 
Let $(E_i, \alpha_{ik}) \in
\twcx_\B^{\pm}\left(\twcx_\B^{\pm}\left(\A\right)\right)$. Then: 
\begin{enumerate}
\item Functors $\cxrow$ and $\cxcol$ in 
Defn.~\ref{defn-the-functor-of-rows-and-the-functor-of-columns}
are well-defined. The data they assign to an object of 
$\twcx_\B^{\pm}\left(\twcx_\B^{\pm}\left(\A\right)\right)$
satisfies the twisted bicomplex condition 
\eqref{eqn-the-twisted-bicomplex-condition} and the finiteness
condition of the sum of its differentials being a morphism in $\B$. 
The data they assign to a morphism satisfies the finiteness 
condition of the sum of its components being a morphism in $\B$. 
\item The following diagram commutes:
\begin{equation}
\begin{tikzcd}[column sep = 2.5cm]
\twcx_\B^{\pm}\left(\twcx_\B^{\pm}\left(\A\right)\right) 
\ar{r}{\conv\conv}
\ar{d}[']{\cxrow \text{ or } \cxcol} 
&
\B. 
\\
\twbicx_\B^{\pm}(\A)
\ar{ur}[']{\conv}
&
\end{tikzcd}
\end{equation}
\item The following diagram commutes:
\begin{equation}
\begin{tikzcd}[column sep = 2.5cm]
\twcx_\B^{\pm}\left(\twcx_\B^{\pm}\left(\A\right)\right) 
\ar{r}{\cxrow}
\ar{d}[']{\cxcol} 
&
\twbicx_\B^{\pm}(\A).
\\
\twbicx_\B^{\pm}(\A)
\ar{ur}[']{\text{Reflect}}
&
\end{tikzcd}
\end{equation}
Here $\text{Reflect}$ is a self-inverse automorphism of $\twbicx_\B^{\pm}(\A)$
which reflects it along the diagonal: a bicomplex $(E_{ij}, \alpha_{ijkl})$
is sent to $(E_{ji}, (-1)^{\delta_{ik}+\delta_{jl}} \alpha_{jilk})$, 
while a morphism $(f_{ijkl})$ is sent to $(f_{jilk})$. 

\item Functor $\cxrow$ restricts to the isomorphism 
$$ \cxrow\colon 
\pretriag_\B^{\pm}\left(\twcx_\B^{\pm}\left(\A\right)\right)
\xrightarrow{\sim}
\twbicx_\B^{\pm,vos}(\A),$$
while $\cxcol$ restricts to the isomorphism 
$$ \cxcol\colon 
\pretriag_\B^{\pm}\left(\twcx_\B^{\pm}\left(\A\right)\right)
\xrightarrow{\sim}
\twbicx_\B^{\pm,hos}(\A).$$
Both functors restrict to the isomorphisms 
$$ 
\cxrow, \cxcol\colon 
\pretriag_\B^{\pm}\left(\pretriag_\B^{\pm}\left(\A\right)\right)
\xrightarrow{\sim}
\twbicx_\B^{\pm,os}(\A).$$
\end{enumerate}
\end{prps}
\begin{proof}
Straightforward computation.  
\end{proof}

The resulting ``monodromy'' $\cxrow^{-1} \circ \cxcol$
of $\pretriag_\B^{\pm}\left(\pretriag_\B^{\pm}\left(\A\right)\right)$
is a non-trivial autoequivalence which takes a complex of complexes
and slices up the resulting bigraded data of objects and differentials 
in the other direction to produce a different complex of complexes out 
of the same data, while sign-twisting purely horizontal and vertical 
differentials.  

\section{Application: $\Ainfty$-structures in monoidal DG categories}
\label{section-ainfty-structures-in-monoidal-dg-categories}

The main application we had in mind for unbounded twisted complexes is 
to reformulate and generalise the definitions of $\Ainfty$-algebras 
and modules \cite[\S2]{Lefevre-SurLesAInftyCategories}: we 
define these structures in an arbitrary DG monoidal category $\A$ 
(or, more generally, a DG bicategory). This disposes with 
the necessity to work explicitly with the operation $m_1$, 
i.e. the differential. 

Traditionally, $\Ainfty$-algebra formalism was defined for 
objects in the DG category $\modk$ of DG complexes of $k$-modules 
with its natural monoidal structure given by the tensor product of complexes
\cite[\S2]{Lefevre-SurLesAInftyCategories}. 
In $\modk$ the internal differential of each object, that is
-- its differential as a complex of $k$-modules, exists 
as a degree $1$ endomorphism of the object. 
It can therefore be a part of the definition of an
$\Ainfty$-algebra or $\Ainfty$-module in $\modk$. 
This is no longer true if we work with an arbitraty monoidal 
DG category $\A$. In $\modA$ the internal differentials of objects
do not appear as their degree $1$ endomorphisms. Moreover, if we
wanted to try and set up $\Ainfty$-formalism to work in $\A$ itself, 
its objects do not possess an internal differential. 

The language of twisted complexes solves both of these problems. 
It implicitly embeds the objects of $\A$ into $\modA$ as
$\Hom$-complexes of $\A$. These do have an internal differential: 
the differential $d_\A$ of $\A$. The twisted complex
condition \eqref{eqn-the-twisted-complex-condition} involves $d_\A$ and 
makes it possible to define an $\Ainfty$-algebra 
or module structure on an object $a \in \A$ while referring explicitly
only to operations $\left\{ m_i \right\}_{i \geq 2}$. 

The resulting definitions all ask for the corresponding bar (or
cobar) construction of the $\Ainfty$-operations to be a twisted complex.
These twisted complexes have to be unbounded, thus necessitating the
theory developed in this paper and its subtleties. We note that in the bar
constructions there is only a finite number of arrows emerging from
each element of the twisted complex. Hence, our definitions of an 
$\Ainfty$-algebra or an $\Ainfty$-module are independent of the ambient
category $\B$ we use to define unbounded twisted complexes. On the
other hand, in cobar constructions this is no longer the case and
the choice of $\B$ becomes crucial. We see another example of these
subtleties coming into play when we consider twisted complexes of
$\Ainfty$-modules in
\S\ref{section-twisted-complexes-of-ainfty-modules}. 

In \S\ref{section-ainfinity-structures-definitions} we give 
the key definitions which are studied further
\cite{AnnoLogvinenko-AinfinityStructuresInDGMonoidalCategoriesAndStrongHomotopyUnitality}. An interested reader should consult \S3.2 of that paper for 
further explanation of the way in which these definitions generalise 
the classical ones in \cite[\S2]{Lefevre-SurLesAInftyCategories}. 

In \S\ref{section-twisted-complexes-of-ainfty-modules} 
we use the twisted bicomplex techniques we developed in
\S\ref{section-twisted-bicomplexes} to prove several theorems about
twisted complexes of $\Ainfty$-modules. We first relate a twisted complex 
of $\Ainfty$-modules to an $\Ainfty$-module structure on the twisted
complex of their underlying objects. This allows us to show
that the DG category $\nodA$ of $\Ainfty$-modules over an
$\Ainfty$-algebra $A$ is strongly pretriangulated (resp.
pretriangulated) if and only if DG monoidal category $\A$ we 
work in is. Hence if we expand $\A$ to $\modA$ with the induced 
monoidal structure
\cite[\S4.5]{GyengeKoppensteinerLogvinenko-TheHeisenbergCategoryOfACategory}
all the categories of $\Ainfty$-modules over all $\Ainfty$-algebras in
it become strongly pretriangulated.  

\subsection{Definitions}
\label{section-ainfinity-structures-definitions}

Throughout this section we assume that DG monoidal category $\A$ we
work with comes with a fixed choice of a monoidal embedding 
$$ \A \hookrightarrow \B $$
into a closed monoidal DG category $\B$ which admits convolutions
of unbounded twisted complexes. Note, that we
can always set $\B = \modA$ with the induced monoidal structure
\cite[\S4.5]{GyengeKoppensteinerLogvinenko-TheHeisenbergCategoryOfACategory},
enlarging our universe if necessary when $\A$ is not small. All 
the unbounded twisted complexes over $\A$ are then defined 
relative to this ambient category $\B$.  

The condition that $\B$ is closed under convolutions of unbounded
twisted complexes can be replaced throughout by $\B$ being closed under 
the convolutions of bounded above twisted complexes and/or bounded below
twisted complexes.  

\begin{defn}
\label{defn-algebra-bar-construction-in-a-monoidal-category}
Let $\A$ be a monoidal DG category, let $A \in \A$ and let 
$\left\{m_i\right\}_{i \geq 2}$ be a collection of degree $2-i$
morphisms $A^i \rightarrow A$. 

The \em (non-augmented) bar-construction $\infbarnaug(A)$ \rm of $A$ 
is the collection of objects $A^{i+1}$ for all $i \geq 0$ each placed
in degree $-i$ and of degree $k-1$ maps 
$d_{(i+k)i}\colon A^{i+k} \rightarrow A^i$ 
defined by
\begin{equation}
\label{eqn-differentials-in-non-aug-bar-construction}
d_{(i+k)i} := (-1)^{(i-1)(k+1)} \sum_{j = 0}^{i-1} (-1)^{jk} \id^{i-j-1}
\otimes m_{k+1} \otimes \id^{j}. 
\end{equation} 

\begin{tiny}
\begin{equation}
\label{eqn-nonaugmented-bar-construction-of-A-m_i}
\begin{tikzcd}[column sep = 2.5cm]
\dots
\ar{r}[']{\begin{smallmatrix}A^3 m_2  - A^2m_2A + \\ + Am_2A^2 - m_2 A^3 \end{smallmatrix}}
\ar[bend left=25]{rr}[description]{A^2m_3 + Am_3A +  m_3 A^2}
\ar[bend left=30]{rrr}[description]{Am_4 - m_4A}
\ar[bend left=35]{rrrr}[description]{m_5}
&
A^4
\ar{r}[']{A^2m_2 - A m_2 A + m_2 A^2}
\ar[bend left=25]{rr}[description]{-Am_3 - m_3A}
\ar[bend left=30]{rrr}[description]{m_4}
& 
A^3
\ar{r}[']{Am_2 - m_2A}
\ar[bend left=25]{rr}[description]{m_3}
&
A^2
\ar{r}[']{m_2}
&
\underset{\degzero}{A}
\end{tikzcd}
\end{equation}
\end{tiny}

\end{defn}

\begin{defn}
\label{defn-ainfty-algebra-in-a-monoidal-category}
Let $\A$ be a monoidal DG category. An \em $\Ainfty$-algebra $(A,m_i)$ \rm 
in $\A$ is an object $A \in \A$ equipped 
with operations $m_i \colon A^i \rightarrow A$ for all $i \geq 2$
which are degree $2-i$ morphisms in $\A$ such that their non-augmented 
bar-construction $\infbarnaug(A)$ is a twisted complex over $\A$. 
\end{defn}

We define morphisms of $\Ainfty$-algebras in $\A$ in a similar way:

\begin{defn} 
\label{defn-the-bar-construction-of-a-morphism-of-ainfty-algebras}
Let $(A, m_k)$ and $(B, n_k)$ be $\Ainfty$-algebras in $\A$. 
Let $(f_i)_{i \geq 1}$ be a collection of degree $1 - i$ 
morphisms $A^i \rightarrow B$. 

The \em bar-construction \rm $\infbar(f_\bullet)$ is the morphism 
$\infbarnaug(A) \rightarrow \infbarnaug(B)$ in $\pretriagmns(\A)$
whose $A^{i+k} \rightarrow B^i$ component is 
$$ \sum_{t_1 + \dots + t_i = i + k} 
(-1)^{\sum_{l=2}^{i}(1-t_l)\sum_{n=1}^l t_n}
f_{t_1}\otimes\ldots \otimes f_{t_i}. $$
\begin{small}
\begin{equation*}
\begin{tikzcd}[column sep = 2.6cm, row sep = 3cm]
\dots
\ar{r}
&
A^4
\ar{r}
\ar{d}[description, pos = 0.80]{f_1f_1f_1f_1}
\ar{dr}[description, pos = 0.60]{f_1f_1f_2 - f_1 f_2 f_1 + f_2 f_1 f_1}
\ar{drr}[description, pos = 0.30]{f_1f_3 + f_2 f_2 + f_3 f_1}
\ar{drrr}[description, pos = 0.10]{f_4}
&
A^3
\ar{r}
\ar{d}[description, pos = 0.80]{f_1f_1f_1}
\ar{dr}[description, pos = 0.60]{-f_1f_2 + f_2 f_1}
\ar{drr}[description, pos = 0.30]{f_3}
&
A^2
\ar{r}
\ar{d}[description, pos = 0.80]{f_1f_1}
\ar{dr}[description, pos = 0.60]{f_2}
&
A
\ar{d}[description, pos = 0.80]{f_1}
\\
\dots
\ar{r}
&
B^4
\ar{r}
&
B^3
\ar{r}
&
B^2
\ar{r}
&
B.
\end{tikzcd}
\end{equation*}
\end{small}
\end{defn}
\begin{defn}
\label{defn-morphism-of-ainfty-algebras}
A morphism $f_\bullet\colon (A,m_k) \rightarrow (B, n_k)$ of 
$\Ainfty$-algebras is a collection 
$(f_i)_{i \geq 1}$ of degree $1 - i$ morphisms $A^i \rightarrow B$ 
whose bar construction is a closed degree $0$ morphism of twisted
complexes. 
\end{defn}

We define left and right $\Ainfty$-modules over such $(A,m_\bullet)$
in a similar way:

\begin{defn}
\label{defn-right-module-bar-construction-in-a-monoidal-category}
Let $(A,m_i)$ be an $\Ainfty$-algebra in a monoidal DG category $\A$. 
Let $E \in \A$ and let $\left\{p_i\right\}_{i \geq 2}$ be a collection 
of degree $2-i$ morphisms $E \otimes A^{i-1} \rightarrow E$. 

The \em right module bar-construction $\infbar(E)$ \rm of $(E,p_i)$
comprises objects $E \otimes A^{i}$ for $i \geq 0$ 
placed in degree $-i$ and degree $1-k$ maps 
$E \otimes A^{i+k-1} \rightarrow E \otimes A^{i-1}$ 
defined by
\begin{scriptsize}
\begin{equation}
\label{eqn-differentials-in-right-module-bar-construction}
d_{(i+k)i} := (-1)^{(i-1)(k+1)} 
\left(\sum_{j = 0}^{i-2} \left( (-1)^{jk} \id^{i-j-1} \otimes m_{k+1} \otimes
\id^{j}\right) + (-1)^{(i-1)k}p_{k+1}\otimes \id^{i-1} \right). 
\end{equation} 
\end{scriptsize}

\begin{tiny}
\begin{equation}
\label{eqn-right-module-bar-construction-of-A-m_i}
\begin{tikzcd}[column sep = 2.4cm]
\dots
\ar{r}[']{\begin{smallmatrix}EA^2 m_2  - EAm_2A + \\ + Em_2A^2 - p_2 A^3 \end{smallmatrix}}
\ar[bend left=25]{rr}[description]{EA m_3 + Em_3A + p_3 A^2 }
\ar[bend left=30]{rrr}[description]{Em_4 - p_4A}
\ar[bend left=35]{rrrr}[description]{p_5}
&
EA^3
\ar{r}[']{EAm_2 - E m_2 A + p_2 A^2}
\ar[bend left=25]{rr}[description]{-Em_3 - p_3A}
\ar[bend left=30]{rrr}[description]{p_4}
& 
EA^2
\ar{r}[']{Em_2 - p_2A}
\ar[bend left=25]{rr}[description]{p_3}
&
EA
\ar{r}[']{p_2}
&
\underset{\degzero}{E}.
\end{tikzcd}
\end{equation}
\end{tiny}
\end{defn}

\begin{defn}
\label{defn-left-module-bar-construction-in-a-monoidal-category}
For $E \in \A$ and a collection 
$\left\{p_i\right\}_{i \geq 2}$ of degree $2-i$
morphisms $ A^{i-1}\otimes E \rightarrow E$,
its \em left module bar-construction $\infbar(E)$ \rm
comprises objects $A^{i}\otimes E$ for all $i \geq 0$ placed
in degree $-i$ and degree $1-k$ maps 
$A^{i+k-1}\otimes E \rightarrow A^{i-1}\otimes E$ 
defined by
\begin{equation}
\label{eqn-differentials-in-left-module-bar-construction}
d_{(i+k)i} := (-1)^{(i-1)(k+1)} 
\left(\sum_{j = 1}^{i-1} \left( (-1)^{jk} \id^{i-j-1} \otimes m_{k+1} \otimes
\id^{j}\right) + \id^{i-1} \otimes p_{k+1} \right). 
\end{equation} 

\begin{tiny}
\begin{equation}
\label{eqn-left-module-bar-construction-of-A-m_i}
\begin{tikzcd}[column sep = 2.4cm]
\dots
\ar{r}[']{\begin{smallmatrix}A^3 p_2  - A^2m_2E + \\ + Am_2AE - m_2 A^2E \end{smallmatrix}}
\ar[bend left=25]{rr}[description]{A^2 p_3 + Am_3E + m_3 AE }
\ar[bend left=30]{rrr}[description]{Ap_4 - m_4E}
\ar[bend left=35]{rrrr}[description]{p_5}
&
A^3E
\ar{r}[']{A^2p_2 - A m_2 E + m_2 AE}
\ar[bend left=25]{rr}[description]{-Ap_3 - m_3E}
\ar[bend left=30]{rrr}[description]{p_4}
& 
A^2E
\ar{r}[']{Ap_2 - m_2E}
\ar[bend left=25]{rr}[description]{p_3}
&
AE
\ar{r}[']{p_2}
&
\underset{\degzero}{E}.
\end{tikzcd}
\end{equation}
\end{tiny}
\end{defn}

\begin{defn}
Let $\A$ be a monoidal DG category and let $(A,m_i)$ be an
$\Ainfty$-algebra in $\A$. A \em right (resp. left) $\Ainfty$-module 
$(E, p_i)$ over $A$ \rm is an object $E \in \A$ and a collection  
$\left\{p_i\right\}_{i \geq 2}$ of degree $2-i$
morphisms $E \otimes A^{i-1} \rightarrow E$ (resp. $A^{i-1} \otimes E
\rightarrow E$) such that $\infbar(E)$ is a twisted complex. 
\end{defn} 

\begin{defn} 
\label{defn-right-module-bar-constructions-of-an-Ainfty-morphism}
Let $\A$ be a monoidal DG category and let $(A,m_i)$ be an
$\Ainfty$-algebra in $\A$.
Let $(E, p_k)$ and $(F, q_k)$ be right $\Ainfty$-modules over $A$ in $\A$. 

A \em degree $j$ morphism \rm $f_\bullet\colon (E,p_k) \rightarrow (F, q_k)$ of 
right $\Ainfty$-$A$-modules is a collection $(f_i)_{i \geq 1}$ of 
degree $j - i + 1$ morphisms $E \otimes A^{i-1} \rightarrow F$. Its
\em bar-construction \rm $\infbar(f_\bullet)$ is 
the morphism $\infbar(E) \rightarrow \infbar(F)$ in $\pretriagmns(\A)$
whose components are  
$$
E \otimes A^{i+k-1} \rightarrow F \otimes A^{i-1}\colon \;
(-1)^{j(i-1)}  f_{k+1}\otimes \id^{i-1}. $$
We illustrate the case when $f_\bullet$ is of odd degree:
\begin{small}
\begin{equation*}
\begin{tikzcd}[column sep = 2.6cm, row sep = 3cm]
\dots
\ar{r}
&
EA^3
\ar{r}
\ar{d}[description, pos = 0.85]{-f_1A^3}
\ar{dr}[description, pos = 0.65]{f_2 A^2}
\ar{drr}[description, pos = 0.30]{-f_3 A} 
\ar{drrr}[description, pos = 0.10]{f_4}
&
EA^2
\ar{r}
\ar{d}[description, pos = 0.85]{f_1 A^2}
\ar{dr}[description, pos = 0.65]{-f_2A}
\ar{drr}[description, pos = 0.30]{f_3}
&
EA
\ar{r}
\ar{d}[description, pos = 0.85]{-f_1A}
\ar{dr}[description, pos = 0.65]{f_2}
&
E
\ar{d}[description, pos = 0.85]{f_1}
\\
\dots
\ar{r}
&
FA^3
\ar{r}
&
FA^2
\ar{r}
&
FA
\ar{r}
&
F.
\end{tikzcd}
\end{equation*}
\end{small}
\end{defn}

The corresponding definition for the left $\Ainfty$-modules differs only
in signs:
\begin{defn} 
\label{defn-left-module-bar-constructions-of-an-Ainfty-morphism}
Let $\A$ be a monoidal DG category and let $(A,m_i)$ be an
$\Ainfty$-algebra in $\A$.
Let $(E, p_k)$ and $(F, q_k)$ be left $\Ainfty$-modules over $A$ in $\A$. 

A degree $j$ morphism $f_\bullet\colon (E,p_k) \rightarrow (F, q_k)$ of 
left $\Ainfty$-$A$-modules is a collection $(f_i)_{i \geq 1}$ of 
degree $j - i + 1$ morphisms $A^{i-1} \otimes E \rightarrow F$. 
Its \em bar-construction \rm $\infbar(f_\bullet)$ is 
the morphism $\infbar(E) \rightarrow \infbar(F)$ in $\pretriagmns(\A)$
whose components are 
$$ A^{i+k-1} \otimes E \rightarrow A^{i-1} \otimes F \colon \;
(-1)^{(j+k)(i-1)}  \id^{i-1} \otimes f_{k+1}. $$
\end{defn}

We define the DG categories of left and right modules over $A$ in the
unique way which makes the left and right module bar constructions 
into faithful DG functors from these categories to $\pretriagmns(\A)$:

\begin{defn}
Let $\A$ be a monoidal DG category and $A$ be an $\Ainfty$-algebra in
$\A$. Define the \em DG category $\nodA$ of
right $\Ainfty$-$A$-modules in $\A$ \rm by:
\begin{itemize}
\item Its objects are right $\Ainfty$-$A$-modules in $\A$,
\item For any $E,F \in \obj \noddinf A$, the complex 
$\homm^\bullet_{\noddinf A}(E,F)$
consists of $\Ainfty$-morphisms $f_\bullet\colon E \rightarrow F$
with their natural grading. The differential and
the composition is defined by composing the corresponding 
twisted complex morphisms. 
\item The identity morphism of $E \in \noddinf A$ is the morphism 
$(f_\bullet)$ with $f_1 = \id_E$ and $f_{\geq 2} = 0$ whose
corresponding twisted complex morphism is $\id_{\infbar(E)}$. 
\end{itemize}
The \em DG category $\Anod$ of left $\Ainfty$-$A$-modules 
in $\A$ \rm is defined analogously.  
\end{defn}

Similar definitions exist for $\Ainfty$-coalgebras and 
$\Ainfty$-comodules, see
\cite[\S6]{AnnoLogvinenko-AinfinityStructuresInDGMonoidalCategoriesAndStrongHomotopyUnitality}. 

\subsection{Twisted complexes of $\Ainfty$-modules}
\label{section-twisted-complexes-of-ainfty-modules}

The notion of an $\Ainfty$-module over an $\Ainfty$-algebra
$(A,m_\bullet)$ in a monoidal DG category $\A$ which we defined in
\S\ref{section-ainfinity-structures-definitions}
differs in several ways from the usual notion which corresponds to 
the case $\A = \modk$. 

One is that the DG-category of usual $\Ainfty$-modules is strongly 
pretriangulated, while in our generality $\nodA$ doesn't have to be. 
In this section we show that $\nodA$ is strongly pretriangulated 
if and only if $\A$ is strongly pretriangulated. 

First, we need to fix our conventions. 
As in \S\ref{section-ainfinity-structures-definitions} we assume 
that our monoidal DG category $\A$ comes with a monoidal
embedding into a closed monoidal DG category $\B$ which has
convolutions of unbounded twisted complexes. Recall that we
can always set $\B = \modA$ with the induced monoidal structure
\cite[\S4.5]{GyengeKoppensteinerLogvinenko-TheHeisenbergCategoryOfACategory},

We define $\twcxub \A$ and $\twcxub \B$ relative to $\B$.  
Thus twisted complexes in $\twcxub \A$ and $\twcxub \B$ can have 
infinite number of differentials and/or morphism components emerge from a
single object, but only if their sum still defines a morphism in $\B$. 
By Lemma \ref{lemma-admits-convolutions-of-unbounded-twisted-complexes},
since $\B$ admits convolutions of unbounded twisted complexes, 
the convolution functor $\twcxub \B \hookrightarrow \B$ is an equivalence. 

Let $(A,m_\bullet)$ be an $\Ainfty$-algebra in $\A$. In
\S\ref{section-ainfinity-structures-definitions} we define the DG
category $\nodA$ of right $\Ainfty$-$A$-modules in $\A$. 
Using the monoidal embedding we can view $(A,m_\bullet)$ as an
$\Ainfty$-algebra in $\B$. We write $\nodrepA$ for the category 
of right $\Ainfty$-$A$-modules in $\B$. Note that we have tautological
embedding $\nodA \hookrightarrow \nodrepA$. 

We now want to define $\twcxub \nodA$. The ``naive'' 
definition gives us a convolution into 
$\modd\text{-}\left(\nodA\right)$, but it isn't the category we want 
to work with. Instead, we have an embedding of $\nodA$ into $\nodrepA$, 
and we want $\nodrepA$ to be the target for the convolution of twisted 
complexes of $\nodA$. 

The category $\nodrepA$ is closed under shifts and
direct sums because $\B$ is. Indeed, $(E,p_\bullet)[n] = (E[n], (-1)^n
p_\bullet)$ and $\oplus_i (E_i, p_{i\bullet}) = (\oplus_i E_i, \sum_i
p_{i\bullet})$. To see that $\sum_i p_{i\bullet}$ define an
$\Ainfty$-module structure on $\oplus_i E_i$, note that since
$\B$ is closed monoidal its monoidal structure commutes 
with infinite direct sums. We thus define both 
$\twcxub \nodA$ and $\twcxub \nodrepA$ relative to $\nodrepA$. 
This yields fully faithful convolution functors from both into
$\modd\text{-}\left(\nodrepA\right)$. 
We now want to show that $\nodrepA$ admits convolutions of unbounded 
twisted complexes, and thus both convolution functors take values 
in $\nodrepA$. 

For this, we prove below that a twisted complex of $\Ainfty$-modules
defines the structure of an $\Ainfty$-module on the
underlying twisted complex of objects of $\A$. But first,  
we define what such structure is. The category $\twcxub \A$ is 
not apriori monoidal as a tensor product of two twisted complexes over $\A$ 
should have as objects direct sums of objects of $\A$. 
These do not apriori exist in $\A$, but they do exist in $\B$. 
Indeed, $\twcxub \B$ is a monoidal category equivalent to $\B$. 

\begin{defn}
Let $\twcxA$ denote $A$ considered as a trivial twisted complex 
in $\twcxub \A$. Define $\nodtwcxA$ to be the full subcategory 
of $\nodrepA$ consisting of the $\Ainfty$-modules whose underlying objects of
$\B$ lie in $\twcxub \A \subseteq \twcxub \B \simeq \B$.  
\end{defn}

Explicitly, an object of $\nodtwcxA$ is a twisted complex 
$(E_i, \alpha_{ij})$ over $\A$ together with degree $2-k$
twisted complex morphisms 
$$ p_k\colon (E_i \otimes A^{k-1}, \alpha_{ij} \otimes \id)
\rightarrow (E_i, \alpha_{ij}) $$ 
such that their right-module bar-construction is a twisted complex 
of twisted complexes $(E_i \otimes A^{k-1}, \alpha_{ij} \otimes \id)$. 

We can now state the main result of this subsection:

\begin{prps}
\label{prps-twcxub-nodA-embeds-into-nod-twcxub-A}
There exist fully faithful embeddings of DG-categories: 
\begin{equation}
\label{eqn-twcxub-nodA-embeds-into-nod-twcxub-A}
\Phi\colon 
\twcxub\left(\nodA\right) 
\hookrightarrow 
\nodtwcxA,
\end{equation}
\begin{equation}
\label{eqn-twcxub-nodrepA-embeds-into-nod-twcxub-repA}
\Phi\colon 
\twcxub\left(\nodrepA\right) 
\hookrightarrow 
\nodrepA.
\end{equation}

These preserve boundedness and one-sidedness of twisted
complexes. We can replace $\twcxub$ with any of $\twcxpls$,
$\twcxmns$, $\twcx$, $\pretriagub$, $\pretriagpls$, $\pretriagmns$, 
or $\pretriag$. 
\end{prps}

Note that the embedding
$\A \hookrightarrow \B$ and the convolution functor $\twcxub \A
\hookrightarrow \B$ induces fully faithful functors from the
LHS and the RHS of \eqref{eqn-twcxub-nodA-embeds-into-nod-twcxub-A}
to those of \eqref{eqn-twcxub-nodrepA-embeds-into-nod-twcxub-repA}. 
Our construction of $\Phi$ in the proof below
ensures that \eqref{eqn-twcxub-nodA-embeds-into-nod-twcxub-A} is 
the restriction of  
\eqref{eqn-twcxub-nodrepA-embeds-into-nod-twcxub-repA}. 

\begin{proof}
The bar-construction functor $\infbar: \nodA \rightarrow \pretriagmns\A$
induces a functor 
$\twcxub(\infbar)\colon \twcxub \nodA
\rightarrow
\twcx^{\pm}\left(\twcx^{\pm}\A\right)$. 
Composing it with $\cxcol$ (see 
Defn.~\ref{defn-the-functor-of-rows-and-the-functor-of-columns})
we get a functor 
$$ 
\cxcol \circ \twcxub(\infbar)\colon 
\twcxub\left(\nodA)\right)
\rightarrow \twbicx_{\B}^{\pm}(\A). $$
Similarly, composing
$\infbar\colon \nodA^{\twcxub(\A)} \rightarrow
\pretriagmns\left(\twcxub(\A)\right)$
with $\cxrow$ gives 
$$ \cxrow \circ \infbar \colon
\nodA^{\twcxub(\A)} \rightarrow \twbicx_{\B}^{\pm}(\A). $$
The functors $\infbar$, $\cxcol$, and $\cxrow$ are injective on 
objects and faithful. Hence so are $\cxcol \circ \twcxub(\infbar)$
and $\cxrow \circ \infbar$.

This proof is based on the observation that the image
of $\cxcol \circ \twcxub(\infbar)$ is mapped into the image of
$\cxrow \circ \infbar$ by the automorphism $\sigma$ of
$\twbicx_{\B}^{\pm}(\A)$ which multiplies every differential
$\alpha_{ijkl}$ and morphism component $f_{ijkl}$ by $(-1)^{ij+kl}$. 
Thus there exists a unique functor $\Phi$
which makes the following square commute:
\begin{equation}
\label{eqn-cxrow-infbar-cxcol-infbar}
\begin{tikzcd}[column sep = 3cm]
\twcxub\left(\nodA\right)
\ar{r}{\cxcol \circ \twcxub(\infbar)}
\ar{d}{\Phi}
&
\twbicx_{\B}^{\pm}(\A)
\ar{d}{\sigma}
\\
\nodA^{\twcxub(\A)} 
\ar{r}{\cxrow \circ \infbar}
&
\twbicx_{\B}^{\pm}(\A).
\end{tikzcd}
\end{equation}

Explicitly, $\Phi$ has the following description. Let
$((E_i, p_{i\bullet}), \alpha_{ij\bullet}) \in
\twcxub\left(\nodA\right)$. Set 
$P_{k+1}\colon (E_i, \alpha_{ij1}) \otimes A^k \to (E_i,\alpha_{ij1})$
to be the morphism of twisted complexes whose components are
$(-1)^{ik}\alpha_{ijk} + \delta_{ij} (-1)^{i(k+1)}p_{i,k+1}$.
Then 
\begin{align}
\label{eqn-definition-of-Phi-on-objects}
\Phi((E_i, \alpha_{ij\bullet})) & = ((E_i,\alpha_{ij1}), P_\bullet)
\\
\label{eqn-definition-of-Phi-on-morphisms}
\Phi((f_{ij\bullet})) & = \left( (-1)^{i\bullet} f_{ij\bullet}\right). 
\end{align}

Taking the above as the definition of $\Phi$ in 
\eqref{eqn-twcxub-nodA-embeds-into-nod-twcxub-A}, one can now verify by 
direct computation that $\Phi$ is well-defined and that it makes 
\eqref{eqn-cxrow-infbar-cxcol-infbar} commute. Its fully faithfullness 
follows immediately from \eqref{eqn-definition-of-Phi-on-morphisms}. 

An identical bicomplex argument applies to 
$\twcxub\left(\nodrepA\right)$ leading to the identical definition of 
$\Phi$ in \eqref{eqn-twcxub-nodrepA-embeds-into-nod-twcxub-repA} 
via the same formulas
\eqref{eqn-definition-of-Phi-on-objects}
and \eqref{eqn-definition-of-Phi-on-morphisms}
which is well-defined and fully faithful for the same reasons.

\end{proof}

It follows that $\nodA$ is pretriangulated if and only if $\A$ is:

\begin{cor}
\label{cor-noddinfA-strongly-pretriangulated-iff-A-is}
The natural embedding 
\begin{equation}
\label{eqn-natural-embedding-nodA-to-twcxub-nodA}
\nodA \hookrightarrow \twcxub\left(\nodA\right) 
\end{equation}
is an equivalence (resp. quasi-equivalence) 
if and only if $\A \hookrightarrow \twcxub\left(\A\right)$ is.
The same holds if $\twcxub$ is replaced by any of its
subcategories $\twcx^\bullet$ or $\pretriag^\bullet$. 
\end{cor}

Note that $\A \hookrightarrow \twcxub\A$ is never an
equivalence. Let $\left\{a_i\right\}$ be a twisted complex
with zero differentials. We define $\twcxub\A$ relative to $\modA$,
so any morphism from some $b \in \A$ to $\left\{a_i\right\}$ only has 
a finite number of non-zero components $b \rightarrow a_i$. So 
it can not have a right inverse. Thus, by above Corollary,
$\nodA$ never admits the convolutions of unbounded twisted 
complexes taken relative to $\nodrepA$. 

\begin{proof}
All arguments in this proof work identically if we replace $\twcxub$ 
with any of its full subcategories $\twcx^\bullet$ or
$\pretriag^\bullet$. 
 
\em ``If'': \rm 
Since \eqref{eqn-natural-embedding-nodA-to-twcxub-nodA}
is fully faithful, it is an equivalence (resp.
quasi-equivalence) if so is its composition 
$\nodA \hookrightarrow \nodtwcxA$ with 
\eqref{eqn-twcxub-nodA-embeds-into-nod-twcxub-A}. 

If $\A \hookrightarrow \twcxub\left(\A\right)$ is an equivalence, 
it is clear that so is $\nodA \hookrightarrow \nodtwcxA$.
If $\A \hookrightarrow \twcxub\left(\A\right)$ is only a quasi-equivalence, 
$\nodA \hookrightarrow \nodtwcxA$ is also a quasi-equivalence, but 
this requires more work.  
Let $((E_i,\alpha_{ij}), p_\bullet) \in \nodtwcxA$. 
By assumption, $(E_i,\alpha_{ij})$ is homotopy 
equivalent in $\twcxub\left(\A\right)$ to some $F \in \A$. 
By the homotopy transfer of structure (Theorem
\ref{theorem-homotopy-transfer-of-structure-for-ainfty-modules-new})
we can transfer the $\Ainfty$-structure $p_\bullet$ from 
$(E_i,\alpha_{ij})$ to $F$. We thus obtain an $\Ainfty$-$A$-module 
$(F,q_\bullet)$ homotopy equivalent to $((E_i,\alpha_{ij}), p_\bullet)$, 
as desired. 

\em ``Only if'': \rm  
The forgetful functor $\nodA \rightarrow \A$ has a right
inverse: the functor $\A \rightarrow \nodA$
which sends any object $a \in \A$ to $(a,p_\bullet)$ with $p_i = 0$ for all $i$
and sends any morphism $f\colon a \rightarrow b$ to $(f_\bullet)$ with 
$f_1 = f$ and $f_i = 0$ for $i > 1$. 
\end{proof}

Since $\B \hookrightarrow \twcxub \B$ is an equivalence, 
the same argument as in 
Cor.~\ref{cor-noddinfA-strongly-pretriangulated-iff-A-is} gives:
\begin{cor}
\label{cor-nod-A^B-has-convolutions-of-unbounded-twisted-complexes}
$\nodrepA$ has convolutions of unbounded twisted complexes. 
\end{cor}
It follows by Lemma
\ref{lemma-admits-convolutions-of-unbounded-twisted-complexes}, that
the convolutions of unbounded twisted complexes in $\nodA$ and
$\nodrepA$ take value in $\nodrepA$, as desired.  

\appendix

\section{Homotopy transfers of structure for $\Ainfty$-modules}
\label{section-homotopy-transfers-of-structure}

In \cite{Markl-TransferringAinftyStronglyHomotopyAssociativeStructures}
Markl described the homotopy transfer of structure for (the usual) 
$\Ainfty$-algebras over a commutative ring. In this section we give its 
analogue for $\Ainfty$-modules over an $\Ainfty$-category. One of the reasons 
to write this down in detail, is to convince ourselves that it works just 
the same for our new notion of $\Ainfty$-modules over an
$\Ainfty$-algebra in a DG monoidal category $\A$ introduced in 
\S\ref{section-ainfty-structures-in-monoidal-dg-categories}. 

Another is that that the bar construction for morphisms of modules,
unlike that for morphisms of algebras, is additive. As result, we can 
give simple explicit formulas for the transfer in terms of 
the bar construction. 

First we describe the homotopy transfer of structure for 
classical $\Ainfty$-modules.  

Let $\A$ be a small $\Ainfty$-category over a commutative ring $k$
in the sense of \cite{Lefevre-SurLesAInftyCategories} 
As usual, denote by $k_\A$ the minimal $k$-linear category with the
same objects as $\A$: $\homm_{k_\A}(a,b)$ is $0$ when $a \neq b$
and $k$ when $a = b$. The category of graded $k_\A$-$k_\A$-bimodules
has a natural monoidal structure given by  $\otimes_k$.
The $\Ainfty$-category $\A$ can be naturally viewed as an $\Ainfty$-algebra 
in this category, and $\Ainfty$-$\A$-modules --- as $\Ainfty$-modules over 
this algebra in the category of graded $k_\A$-modules. 

For our purposes, it is more natural to consider $\A$ to be an $\Ainfty$-algebra 
in the category $k_\A$-$\modd$-$k_\A$ of differentially graded 
$k_\A$-$k_\A$-bimodules. In other words, what is usually known 
as the operation $m_1$ becomes the intrinsic data of the differential 
of the DG-$k_\A$-$k_\A$-bimodule $\A$. The structure $m_\bullet$
of an $\Ainfty$-algebra on this DG-bimodule consists then of
$k_\A$-$\modd$-$k_\A$ maps 
$$ m_i\colon \A^{\otimes_k i} \rightarrow  \A \quad \quad \quad \quad,
\deg(m_i) = 2 - i,\;  i \geq 2. $$ 
Write $\infbar \A$ for the \em bar construction \rm of $\A$
\cite[\S1.2.2]{Lefevre-SurLesAInftyCategories}. It is a DG-coalgebra in 
$k_\A$-$\modd$-$k_\A$ equal as a graded coalgebra to the free
tensor coalgebra $\bigoplus_{i = 0}^\infty \A^{\otimes_k i}$. 

Let $P$ be a right $\Ainfty$-module over $\A$. We consider it, again, 
as a DG-$k_\A$-module with an $\Ainfty$-structure $\mu_\bullet$ given 
by $\modd$-$k_\A$ maps 
$$ \mu_i: P \otimes_k \A^{\otimes_k (i-1)} \rightarrow P, 
\quad \quad \quad \quad \deg(\mu_i) = 2 - i,\; i \geq 2. $$
Write $\infbar P$ for the  \em bar construction \rm of $\A$
\cite[\S2.3.3]{Lefevre-SurLesAInftyCategories}. It is a 
DG-$\infbar \A$-comodule equal as a graded comodule 
to the free comodule $P \otimes_k \infbar \A$. 

Suppose that there is another $k_\A$-module $Q$ and two morphisms 
$$ f\colon  P \rightarrow  Q, $$
$$ g\colon  Q \rightarrow  P, $$
in $\modd$-$k_\A$ such that $gf = \id + dh$ for some degree $-1$ map $h: P \to P$. 

\begin{theorem}
\label{theorem-homotopy-transfer-of-structure-for-ainfty-modules}
A homotopy transfer of $A_\infty$-structure from $P$ to $Q$ exists: 
\begin{enumerate} 
\item A structure $\nu_\bullet$ of 
a right $A_\infty$-module over $A$ on $Q$;
\item A closed degree zero map of $A_\infty$-modules $\phi_\bullet: P \to Q$ extending $f$;
\item A closed degree zero map of $A_\infty$-modules $\psi_\bullet: P \to Q$ extending $g$;
\item A degree $-1$ map of $\Ainfty$-modules $H_\bullet\colon P \rightarrow P$
extending $h$, such that 
$$\psi_\bullet \circ \phi_\bullet = \id + d(H_\bullet).$$
\end{enumerate}
\end{theorem}

To prove this, we use the bar-construction 
as a fully faithful embedding of the DG-category of $\Ainfty$-$\A$-modules 
into the DG-category of DG-$\infbar \A$-comodules. See 
\cite[\S2.3.3]{Lefevre-SurLesAInftyCategories}. We give a brief
summary. Let $E$ and $F$ be two $\Ainfty$-$\A$-modules. 
Let $x_\bullet\colon E \rightarrow F$ be a degree $n$ morphism of 
$\Ainfty$-modules. By definition, it is an arbitrary collection of
degree $n+1-i$ maps 
$$ x_i: E \otimes_k \A^{\otimes_k (i-1)} \rightarrow F
\quad \quad \quad \quad i \geq 1, $$
in $\modd$-$k_\A$. Such collection is equivalent to the data
of a degree $n$ $\modd$-$k_\A$ map
$$ x_\bullet\colon \infbar E \rightarrow F, $$
because as a graded module $\infbar E$ is just $E \otimes_k \infbar \A$. 
By universal property of free graded comodules, there is a bijective
correspondence of $k_\A$-module morphisms
$$ \infbar E \rightarrow F, $$
with the morphisms of DG-$\infbar \A$-comodules 
$$ \infbar E \rightarrow \infbar F. $$
It sends $x_\bullet$ to the map 
\begin{equation}
\label{eqn-definition-of-the-bar-construction-of-Ainfty-morphism}
\bar{x}\colon E \otimes_k \infbar \A \xrightarrow{\id \otimes \Delta}
E \otimes_k \infbar \A \otimes_k \infbar \A \xrightarrow{x_\bullet \otimes \id}
F \otimes_k \infbar \A,
\end{equation}
and, conversely, sends any morphism $\bar{x}$ of DG-$\infbar \A$-comodules to
$$ x_\bullet\colon  E \otimes_k \infbar \A \xrightarrow{\bar{x}} F \otimes_k \infbar \A 
\xrightarrow{\id \otimes \epsilon} F. $$
Here $\Delta$ and $\epsilon$ is are the comultiplication and the counit of $\infbar \A$. 

\begin{proof}
By our convention, the $\Ainfty$-structure $\mu_\bullet$ on $P$ is a collection 
of $\mu_i$ for $i \geq 2$. View it as the data of a degree one
$\Ainfty$-morphism $\mu_\bullet\colon P \rightarrow P$ with $\mu_1 = 0$.  
The differential on the bar construction $\infbar P$ is 
$$ \bar{\mu} + \dnat, $$
where $\bar{\mu}$ is the bar construction of the $\Ainfty$-morphism $\mu_\bullet$ 
as per \eqref{eqn-definition-of-the-bar-construction-of-Ainfty-morphism}
and $\dnat$ is the natural differential on the tensor product $P \otimes_k \infbar \A$. 
We then have
$$ (\bar{\mu} + \dnat)^2 = \bar{\mu}^2 + \bar{\mu} \circ \dnat + \dnat
\circ \bar{\mu} = \bar{\mu^2} + d(\bar{\mu}), $$
and therefore  
$$\dnat(\bar{\mu}) = - \bar{\mu}^2. $$
Here $\dnat(-)$ denotes the differentiation as 
an endomorphism of $P \otimes_k \infbar A$, as opposed as an endomorphism of $\infbar P$. 

Now define 
$$ \bar{\rho} = \bar{\mu} + \bar{\mu} \bar{h} \bar{\mu} + 
\bar{\mu} \bar{h} \bar{\mu} \bar{h} \bar{\mu} + ... $$
where $\bar{h}$ is the bar construction of $h$ 
viewed as a strict $\Ainfty$-morphism. Note that
\begin{align*}
\overline{\rho} \overline{h} \overline{\mu}  = 
    \overline{\mu} \overline{h} \overline{\rho} & = 
    \overline{\rho} - \overline{\mu}, \\ 
\dnat(\bar{\rho}) & = \bar{\rho}^2 - \bar{\rho} \overline{dh} \bar{\rho}. 
\end{align*}
Since 
$$ \dnat(\bar{h}) = \overline{dh} = \overline{g}\overline{f} -
\id, $$  
we conclude that  
$$ \dnat(\bar{\rho}) = - \bar{\rho}\bar{g}\bar{f}\bar{\rho}. $$

Define the DG-$\infbar \A$-module morphism 
$\bar{\nu}\colon Q \otimes_k \infbar \A \rightarrow  Q \otimes_k \infbar \A$ by 
$$ \bar{\nu} = \bar{f} \bar{\rho} \bar{g}. $$
Since $f$ and $g$ are closed of degree $0$, we have 
$$ \dnat(\bar{\nu}) = \bar{f} \dnat(\bar{\rho}) \bar{g} 
= -  \bar{f} \bar{\rho} \bar{g} \bar{f} \bar{\rho} \overline{g} = 
- \bar{\nu}^2. $$
We therefore have
$$ (\bar{\nu} + \dnat)^2 = 0, $$
so $\bar{\nu} + \dnat$ defines a new differential on $Q \otimes_n \infbar \A$.
Let $\nu_\bullet$ be the corresponding structure of $\Ainfty$-$\A$-module on $Q$. 

Define next 
\begin{align*}
\bar{\phi} &= \bar{f}(\id +  \bar{\rho}\bar{h} ), \\
\bar{\psi} &= (\id + \bar{h}\bar{\rho}) \bar{g}, \\
\bar{H} &= \bar{h}(\id + \bar{\rho}\bar{h}) = (\id +
\bar{h}\bar{\rho})\overline{h}. 
\end{align*}

We have:
\begin{align*}
d\overline{f} = (\overline{\nu} + \dnat) f - f(\overline{\mu} + \dnat)
= \overline{\nu}\overline{f} - \overline{f}\overline{\mu} + \dnat(f)  
= \overline{f}\overline{\rho}\overline{g}\overline{f} - \overline{f}\overline{\mu}
= \overline{f}(\overline{\rho}\overline{g}\overline{f} - \overline{\mu}).
\end{align*}
and similarly:
\begin{align*}
d\overline{g} &= 
% \overline{\mu}\overline{g} -
% \overline{g}\overline{\nu} + \dnat(g) = 
\dots 
= (\overline{\mu} - 
\overline{g}\overline{f}\overline{\rho})\overline{g}, 
\\
d\overline{h} & = \dots 
= \overline{\mu}\overline{h} + \overline{h}\overline{\mu} +
\overline{g}\overline{f} - \id,
\\
d(\id + \overline{\rho}\overline{h}) &=
% \overline{\mu}(\id + \overline{\rho}\overline{h}) -
% (\id + \overline{\rho}\overline{h})\overline{\mu} + 
% \dnat(\id + \overline{\rho}\overline{h}) 
% =
% \\
% & = \overline{\mu}\overline{\rho}\overline{h} -
% (\overline{\rho} - \overline{\mu}) -
% \overline{\rho}\overline{g}\overline{f}\overline{\rho}\overline{h}
% - \overline{\rho}(\overline{g}\overline{f} - \overline{\id})
% = 
% \\ 
% & = \overline{\mu}\overline{\rho}\overline{h} 
% + \overline{\mu} -
% \overline{\rho}\overline{g}\overline{f}\overline{\rho}\overline{h}
% - \overline{\rho}\overline{g}\overline{f} 
% = 
% \\ 
\dots = \left(\overline{\mu} - \overline{\rho} \overline{g}\overline{f}\right)
\left(\id + \overline{\rho}\overline{h}\right),
\\
d(\id + \overline{h}\overline{\rho}) &=
\dots = \left(\overline{\mu} - \overline{g}\overline{f} \overline{\rho} \right)
\left(\id + \overline{h} \overline{\rho}\right),
\end{align*}

We thus finally compute
\begin{align*}
d(\bar{\phi}) & =
d\left(\overline{f}\right)\left(\overline{\rho}\overline{h} + \id\right) 
+ 
\overline{f}d\left(\overline{\rho}\overline{h} + \id\right) = 
\\
& = \overline{f}\left(\overline{\rho}\overline{g}\overline{f} -
\overline{\mu}\right)(\overline{\rho}\overline{h} + \id) 
+ 
\overline{f}
\left(\overline{\mu} - \overline{\rho}\overline{g}\overline{f}\right)
\left(\id + \overline{\rho}\overline{h}\right)
=  0, 
\\
d(\bar{\psi}) & = \dots = 0, 
\\
d(\overline{H}) & =
% d\overline{h}\left(\id + \overline{\rho}\overline{h}\right)
% - \overline{h}d\left(\id + \overline{\rho}\overline{h}\right) = \\
% & = 
% \left(\overline{\mu}\overline{h} + \overline{h}\overline{\mu} +
% \overline{g}\overline{f} - \id\right)\left(\id + \overline{\rho}\overline{h}\right)
% - \overline{h} \left(\overline{\mu} - \overline{\rho} \overline{g}\overline{f}\right)
% \left(\id + \overline{\rho}\overline{h}\right) = 
% \\ 
% & =
% \left(\id + \overline{h}\overline{\rho}\right)\overline{g}\overline{f} 
% \left(\id + \overline{\rho}\overline{h}\right)
% -\left(\id - \overline{\mu}\overline{h}\right)\left(\id +
% \overline{\rho}\overline{h}\right) = 
% \\
% & = 
\dots = 
\overline{\psi}\overline{\phi} - \id , 
\end{align*}
as desired. 
\end{proof}

It follows that the transfer of structure across a homotopy
equivalence produces a homotopy equivalent $\Ainfty$-module:
\begin{cor}
If $f$ and $g$ are mutually inverse homotopy equivalences, 
then the $\Ainfty$-$\A$-module $(Q, \nu_\bullet)$
obtained from $(P,\mu_\bullet)$ 
by the homotopy transfer of structure along $(f,g)$ is
homotopy equivalent to $(P,\mu_\bullet)$. 
\end{cor}
\begin{proof}
As part of the homotopy transfer of structure constructed in Theorem
\ref{theorem-homotopy-transfer-of-structure-for-ainfty-modules}, 
we have obtained a closed degree zero $\Ainfty$-morphism 
$$ \phi_\bullet\colon (P,\mu_\bullet) \rightarrow (Q,\nu_\bullet) $$
extending $f$, i.e. $\phi_1 = f$. Thus $\phi_1$ is a homotopy
equivalence of DG-$k_\A$-modules, and therefore $\phi_\bullet$ is a homotopy 
equivalence of $\Ainfty$-$\A$-modules \cite[2.4.1.1]{Lefevre-SurLesAInftyCategories}. 
\end{proof}

The method we used to prove Theorem
\ref{theorem-homotopy-transfer-of-structure-for-ainfty-modules}
can be easily applied to the notion of $\Ainfty$-algebras and 
$\Ainfty$-modules in a DG monoidal category introduced in 
\S\ref{section-ainfty-structures-in-monoidal-dg-categories}:
\begin{theorem}
\label{theorem-homotopy-transfer-of-structure-for-ainfty-modules-new}
Let $(A,m_\bullet)$ be an $\Ainfty$-algebra in a monoidal DG category $\A$. 
Let $(a,p_\bullet) \in \nodA$. Let $b$ be an object of $\A$. 
Suppose there exist morphisms  
$f\colon a \rightarrow b$ and $g\colon b \rightarrow a$ in $\A$
such that $gf = \id + dh$ for some degree $-1$ morphism $h\colon a  
\rightarrow a$.  

Then a homotopy transfer of $A_\infty$-structure from $(a,p_\bullet)$ to $b$ exists: 
\begin{enumerate} 
\item A structure $q_\bullet$ of an $A_\infty$-$A$-module on $b$; 
\item A closed degree $0$ map of $\Ainfty$-$A$-modules 
$\phi_\bullet: (a,p_\bullet) \to (b,q_\bullet)$ extending $f$;
\item A closed degree $0$ map of $A_\infty$-$A$-modules
$\psi_\bullet: (b,q_\bullet) \to (a,p_\bullet)$ extending $g$;
\item A degree $\text{-}1$ map of $\Ainfty$-$A$-modules 
$H_\bullet\colon (a,p_\bullet) \rightarrow (a,p_\bullet)$
extending $h$ with 
$$\psi_\bullet \circ \phi_\bullet = \id + d(H_\bullet).$$
\end{enumerate}
\end{theorem}
\begin{proof}
For $\Ainfty$-$A$-modules, the bar construction was defined in 
\S\ref{section-ainfinity-structures-definitions}
as a DG-functor
$$ \infbar\colon \noddinf(T) \rightarrow \pretriagmns(\A). $$ 
Take the data $p_\bullet$ of the $\Ainfty$-$T$-module structure on
$a$ and consider it as the data of a morphism of
$\Ainfty$-$A$-modules with $p_1 = 0$. Let 
$\bar{p}$ be its bar construction:
$$ \bar{p}\colon \infbar(a,p_\bullet) \rightarrow
\infbar(a,p_\bullet). $$

Consider now two twisted complexes: $\infbar(a,p_\bullet)$ 
and $\infbar(a,0)$. They have the same objects. 
Denote by $\dnat(-)$ the operation of
differentiating an endomorphism of $\infbar(a,p_\bullet)$ as 
if it was an endomorphism of $\infbar (a,0)$. 
Since the differential on $\infbar(a,p_\bullet)$ is the sum 
of the differential on $\infbar(a,0)$ and $\bar{p}$, we have
$$ \dnat(\bar{p}) = - \bar{p}^2. $$

We can now proceed in the same way and with the same computations 
as in the proof of Theorem
\ref{theorem-homotopy-transfer-of-structure-for-ainfty-modules}, 
only with all bar constructions being twisted complexes instead 
of DG-comodules. 
\end{proof}
\begin{cor}
If $f$ and $g$ are mutually inverse homotopy equivalences in $\A$, 
then the $\Ainfty$-$A$-module $(b,q_\bullet)$
obtained from $(a,p_\bullet)$ by the homotopy transfer of structure 
along $(f,g)$ is homotopy equivalent to $(a,p_\bullet)$. 
\end{cor}

\bibliography{references}
\bibliographystyle{amsalpha}
\end{document}